\documentclass[reqno]{amsart}
\usepackage{amsmath,amsthm,latexsym,amssymb,hyperref, amsfonts}
\usepackage{fullpage}
\usepackage{stmaryrd}
\usepackage{eucal}
\usepackage{mathrsfs}
\usepackage[all]{xy}
\usepackage{color}
\usepackage{tikz}
\usepackage{url}
\usepackage{tensor}
\usepackage{comment}
\usepackage{tikz-cd}
\usepackage{enumerate}
\usepackage{enumitem}
\usepackage{mathrsfs}
\usepackage{dutchcal}
\usetikzlibrary{matrix,arrows,decorations.pathmorphing}
\title[A Monoidal Dold-Kan Correspondence For Comodules]{A Monoidal Dold-Kan Correspondence For Comodules}

\author[Maximilien P\'eroux]{Maximilien P\'eroux}

\theoremstyle{definition}
\newtheorem{defi}{Definition}[section]

\newtheorem{ex}[defi]{Example}
\newtheorem{rem}[defi]{Remark}
\newtheorem{nota}[defi]{Notation}

\numberwithin{equation}{section}

\theoremstyle{plain}
\newtheorem{thm}[defi]{Theorem}
\newtheorem{prop}[defi]{Proposition}
\newtheorem{lem}[defi]{Lemma}
\newtheorem{cor}[defi]{Corollary}

\renewcommand{\o}{\otimes}
\newcommand{\bI}{\mathbb I}
\newcommand{\I}{\mathbb I}

\newcommand{\nnn}{\NN^C}

\newcommand{\wt}[1]{\widetilde{#1}}

\newcommand{\id}{\mathsf{id}}

\newcommand{\bJ}{\mathbb{J}}

\newcommand{\Q}{\mathbb{Q}}

\newcommand{\Smash}{\wedge}
\newcommand{\sm}{\wedge}

\renewcommand{\r}{\rightarrow}

\renewcommand{\ker}{\mathsf{ker}}
\newcommand{\im}{\mathsf{im}}
\newcommand{\coker}{\mathsf{coker}}

\newcommand{\C}{\mathsf{C}}
\newcommand{\D}{\mathsf{D}}
\newcommand{\op}{^\mathsf{op}}
\newcommand{\M}{\mathsf{M}}

\newcommand{\bn}{\mathbb{N}}

\newcommand{\nn}{\mathsf{N}}

\newcommand{\X}{\mathsf{P}}

\newcommand{\Z}{\mathsf{Q}}
\newcommand{\bz}{\mathbb{Z}}
\newcommand{\Post}{\mathsf{Post}}
\newcommand{\ds}{\displaystyle}

\newcommand{\llp}{\mathsf{LLP}}

\newcommand{\ccotens}{{{\square}_C}}
\newcommand{\cotens}{\square}

\newcommand{\ii}{\mathcal{P}}
\newcommand{\iio}{{\mathcal{P}_{\oplus}}}
\newcommand{\jj}{\mathcal{Q}}

\newcommand{\iii}{{\ii^{\geq 0}}}
\newcommand{\iiio}{{\ii_\oplus^{\geq 0}}}
\newcommand{\jjj}{{\jj^{\geq 0}}}

\newcommand{\iid}{{\ii^\Delta}}
\newcommand{\iido}{{\ii_\oplus^\Delta}}
\newcommand{\jjd}{\jj^\Delta}

\newcommand{\NN}{\mathsf{N}}
\newcommand{\N}{\mathscr{N}}

\newcommand{\A}{\mathsf{A}}

\newcommand{\holim}{\mathsf{holim}}
\renewcommand{\lim}{\mathsf{lim}}
\newcommand{\colim}{\mathsf{colim}}
\newcommand{\collim}[1]{\underset{#1}{\colim}}
\newcommand{\llim}[1]{\underset{#1}{\lim}}

\newcommand{\smod}{\mathsf{sMod}}
\newcommand{\smodk}{\mathsf{sMod}_\k}

\newcommand{\emm}{K(V,n)}
\newcommand{\emmm}{P{K(V,n)}}

\newcommand{\cho}{\mathsf{Ch}_R^{\geq 0}}

\newcommand{\Hom}{\mathsf{Hom}}

\newcommand{\pull}{\arrow[dr, phantom, "\lrcorner", very near start]}
\newcommand{\push}{\arrow[ul, phantom, "\ulcorner", very near start]}

\renewcommand{\k}{\Bbbk}
\newcommand{\ch}{\mathsf{Ch}}
\newcommand{\chf}{{\mathsf{Ch}_\k}}
\newcommand{\chfo}{{\ch_\k^{\geq 0}}}

\newcommand{\comod}{\mathsf{CoMod}}

\newcommand{\ccoalg}{\mathsf{CoCAlg}}

\newcommand{\W}{\mathsf{W}}

\newcommand{\comodinf}{\mathcal{CoMod}}
\newcommand{\Dinf}{\mathcal{D}}

\newcommand{\Cinf}{\mathcal{C}}

\tikzset{
    labl/.style={anchor=south, rotate=-32, inner sep=.9mm}
}

\tikzset{
    labll/.style={anchor=south, rotate=32, inner sep=.9mm}
}

\begin{document}

\address{Department of Mathematics, Michigan State University, 619 Red Cedar Road, East Lansing, MI 48824, USA}
    \email{peroux@msu.edu}

\subjclass [2020] {16T15, 18N40, 18G31, 18G35, 55U15} 
   
\keywords{Dold-Kan correspondence, comodule, fibrantly generated, model category, Postnikov tower, $\infty$-category.}

\begin{abstract}  
Cofibrantly generated model categories are generalizations of CW-approximations which provide an inductive cofibrant replacement. We provide examples of inductive fibrant replacements constructed as Postnikov towers for simplicial and differential graded comodules. Our main application is to show that simplicial comodules and connective differential graded comodules are Quillen equivalent and their derived cotensor products correspond. We deduce that the rational $A$-theory of a simply connected space $X$ is equivalent to the $K$-theory of perfect chain complexes with a $C_*(X; \mathbb{Q})$-comodule structure.
\end{abstract}

\maketitle


\section{Introduction}

The \emph{Dold-Kan correspondence} established in \cite{doldkan1, doldkan2} is an equivalence between the category $\smod_R$ of simplicial $R$-modules and the category $\cho$ of non-negative chain complexes over a commutative ring $R$. This has been generalized to any abelian category, or even any stable $\infty$-category \cite{lurie1}. The \emph{norma\-lization} functor $\NN\colon\smod_R\rightarrow \cho$ is lax symmetric monoidal with respect to the graded tensor product of chains, the levelwise tensor product of simplicial modules and the\emph{ Eilenberg-Zilber map}. The inverse equivalence $\Gamma\colon\cho\r\smod_R$ also has a (non-symmetric) lax monoidal structure coming from the \emph{Alexander-Whitney map}. The categories are not equivalent as symmetric monoidal categories. Nevertheless, the homotopy categories of $\smod_R$ and $\cho$ have isomorphic symmetric monoidal structures when endowed with their respective \emph{derived tensor product}. Formally, it was shown in \cite{monmodSS} that there is a weak symmetric monoidal Quillen equivalence between the symmetric monoidal model categories of $\smod_R$ and $\cho$. The equivalence lifts to the associated categories of algebras and modules.
This equivalence was crucial in \cite{hzalgshipley} to show that homotopy coherent $H\mathbb{Z}$-algebras in spectra  are equivalent to differential graded algebras. Similar results hold for the cosimplicial case \cite{vosimplicial} and the commutative case \cite{richtershipley}.
It is important to notice that at the level of model categories, the results of \cite{monmodSS} do not imply any Quillen equivalences on model categories of coalgebras or comodules. In fact, in \cite{sore3}, it was shown that the Dold-Kan correspondence does not lift to a Quillen equivalence between coassociative and counital coalgebras. 
Differential graded coalgebras are of great important in rational homotopy theory \cite{rational, joseph} as they model Lie algebras.

 In \cite[4.1]{coalginDK}, we proved that the underlying $\infty$-categories of $\smod_R$ and $\cho$ (in the sense of \cite[1.3.4.15, 4.1.7.4]{lurie1}) are equivalent as symmetric monoidal $\infty$-categories. In particular, the $\infty$-categories of homotopy coherent coalgebras in $\smod_R$ and $\cho$ are equivalent. The same holds for homotopy coherent comodules in these categories.
However, unlike the case of algebras and modules, it is more challenging to represent homotopy coherent coalgebras and comodules via a model category of coassociative and counital coalgebra, see \cite{perouxshipley, coalginDK}.
In \cite[1.1]{connectivecomod}, for $R=\k$ a commutative ring with global dimension zero, and $C$ a simply connected strictly cocommutative and coassociative coalgebra in $\chfo$, we proved that homotopy coherent comodules in $\chfo$ over $C$ correspond to strictly coassociative comodules in $\chfo$. The\emph{ cotensor product} of comodules equalize the coactions. As $C$ is always flat, the cotensor remains a comodule. The category $\comod_C(\chfo)$ of connective differential graded comodules over $C$ is endowed with a model structure in which weak equivalences are quasi-isomorphisms between comodules. The model structure is compatible with respect to the cotensor product. This allowed us to endow a symmetric monoidal structure with the \emph{derived cotensor product} on the underlying $\infty$-category of $\comod_C(\chfo)$, see \cite[1.2]{connectivecomod}. Similar arguments hold for simplicial comodules over a simply connected simplicial coalgebra. Although homotopy coherent simplicial comodules correspond to homotopy coherent connective differential graded comodules, it is not immediate that their derived cotensor product correspond. Our main result in this paper shows that they are indeed equivalent.

\begin{thm}[{Theorem \ref{thm: dold-kan correpsondance for Comodules}, Theorem \ref{thm: comonoidal Dold-Kan comomules eq}}]\label{Main theorem}
Let $\k$ be a  commutative ring with global dimension zero. Let $C$ be a simply connected cocommutative differential graded coalgebra over $\k$. Then there is a Quillen equivalence:
\[
\begin{tikzcd}[column sep=large]
\comod_C(\chfo) \ar[shift left=2]{r}{\Gamma} & \ar[shift left=2]{l}{\nn^C}[swap]{\perp} \comod_{\Gamma(C)}(\smodk),
\end{tikzcd}
\]
compatible with the derived cotensor products. In particular, the $\infty$-categories of homotopy coherent simplicial comodules over $\Gamma(C)$ and homotopy coherent connective differential graded comodules over $C$ are equivalent as symmetric monoidal $\infty$-categories.
\end{thm}

One of the consequences of this result is an algebraic description of rational $A$-theory of spaces.
Given a space $X$, recall that $A(X)$ is defined to be the $K$-theory of the Waldhausen category of finite retractive spaces over $X$ with cofibrations and weak equivalences created in simplicial sets \cite{waldy}.
For any generalized reduced homology $\mathcal{E}_*$, let $A(X; \mathcal{E}_*)$ be the Waldhausen $K$-theory of the same category with the same cofibrations, but replacing the usual weak equivalences of simplicial sets by $\mathcal{E}_*$-equivalences. 
Hess-Shipley showed in \cite[1.3]{HSWald} that $A(X;\mathcal{E}_*)$ is equivalent to the $K$-theory of $X_+$-comodules in pointed spaces, localized at $\mathcal{E}_*$. 
Thanks to the result above, we are able to provide a description of the rational $A$-theory of a space $X$ as the $K$-theory of perfect chain complexes with a comodule structure over the singular chain complex $C_*(X; \mathbb{Q})$.

\begin{thm}[{Corollary \ref{cor: rational A-theory}}]
For any simply connected simplicial set $X$, there is a natural equivalence of $K$-theory spectra:
\[
A(X;  H\mathbb{Q}_*)\simeq K(\comod_{C_*(X; \mathbb{Q})}^\textup{perf}).
\]
\end{thm}

A future project will extend the trace method of \cite[1.4]{cothhshadow} to provide a trace into the coHochschild homology $K(\comod_{C_*(X; \mathbb{Q})}^\textup{perf})\rightarrow \mathsf{coHH}(C_*(X; \mathbb{Q}))$, as introduced by Hess-Parent-Scott \cite{coHH}, which relates to the usual Dennis trace if $X$ is finite by \cite{ozgur}.

To prove our main theorem, we dualize the methods in \cite{monmodSS}. Following \cite{hess1} and \cite{left1}, we use the notion of \emph{fibrantly generated} model categories. CW-complexes provide an inductive cofibrant replacement for spaces. This was generalized into the notion of cofibrantly generated model categories. We instead provide \emph{Postnikov towers} as an inductive fibrant replacement. Because of the lack of ``cosmallness" in practice, we cannot apply the dual of the small object argument. Instead we provide ad-hoc \emph{Postnikov presentations}: (acyclic) fibrations are retracts of maps built from pullbacks and transfinite towers of generating (acyclic) fibrations. 
Moreover, the cotensor product on comodules behaves well with fibrant objects instead of cofibrant objects. We introduce the notion of \emph{fibrant-compatible} symmetric monoidal model categories that allows us to derive the cotensor product and obtain a symmetric monoidal structure on the homotopy category and the underlying $\infty$-category. Dualizing \cite{monmodSS}, we construct \emph{weak opmonoidal Quillen equivalences} that induce an equivalence of the derived cotensor product.


\subsubsection*{On the cocommutativity condition} For the sake of simplicity, we focus in this paper on coalgebras that are cocommutative. One can instead consider non-cocommutative coalgebras and then consider bicomodules instead. The cotensor product provides then a non-symmetric monoidal structure and all the results in this paper remain valid. We provide further details in \cite{cothhshadow}.

\subsubsection*{On the global dimension zero condition} Throughout this paper, we work exclusively with $\k$ a commutative ring with global dimension zero. There are several reasons this condition is imposed. First, it induces that every object is cofibrant and fibrant in chain complexes and simplicial modules, and the projective and injective model structures are equal. In particular, the model structure on comodules is left-induced from a nice monoidal model category. In \cite{left2}, it was shown that the model structures for comodules are left-induced from injective model structures which are in general not monoidal model categories. This creates several issues to understand the induced homotopy theory on comodules. Moreover, as every module is flat, the tensor product preserves finite limits. This allows us to understand finite limits of comodules, which is essential for our main theorem.
Furthermore, the cotensor product of comodules is \emph{not} a comodule unless the coalgebra is flat.

\subsubsection*{On the simply connected condition}
The simply connected requirement on the coalgebra is crucial for Theorem \ref{thm: cocell presentation of comodules over simply connected} and thus for Theorem \ref{Main theorem}. It allows us to inductively build our fibrant replacement. 
Moreover, although we make no use of this fact in this paper, as $C$ is a simply connected coalgebra, then the derived category of comodules over $C$ is equivalent to the derived category of modules over its cobar complex $\Omega C$, see \cite{leo}. This Koszul duality does not come from a Quillen equivalence between model structures and results in this paper are not dual from the model structure on modules over $\Omega C$. 

\subsection*{Outline} In Section \ref{Sec: terminology}, we introduce all the dual notions and methods of \cite{monmodSS}: fibrantly generated model categories, Postnikov presentations of a model category, left-induced model categories, fibrant-compatible monoidal structure on model categories and weak opmonoidal Quillen equivalences between them. In Section \ref{Sec: example on chains}, we provide an example of a Postnikov presentation for unbounded chain complexes. It provides inductive methods to compute homotopy limits of chain complexes. In Section \ref{Sec: example on comodules}, we recall some of the results in \cite{connectivecomod} and show that comodules also admit an efficient Postnikov presentation. Finally, we apply our methods to prove our main result on Dold-Kan correspondence for comodules and $A$-theory in Section \ref{section: dold-kan correspondence}.

\subsection*{Acknowledgment}
The results here are part of my PhD thesis \cite{phd}, and as such, I would like to express my gratitude to my advisor Brooke Shipley for her help and guidance throughout the years. Special thanks to Kathryn Hess for many fruitful conversations. I would also like to thank Haldun Özgür Bayındır, Maxine Calle and David Chan for their help.

\subsection*{Notation} We present here the notation and terminology used throughout this paper.

\begin{enumerate}
\item Given a model category $\M$ with a class of weak equivalences $\W$, we denote by $\M_c$ and $\M_f$ its subcategories spanned by cofibrant and fibrant objects respectively. We denote  the\emph{ underlying $\infty$-category} of $\M$ by $\N(\M_c)[\W^{-1}]$ as in \cite[1.3.4.15]{lurie1}, or simply $\N(\M)[\W^{-1}]$ if every object is cofibrant. If $\M$ is a symmetric monoidal model category (in the sense of \cite[4.2.6]{hovey}), then  $\N(\M_c)[\W^{-1}]$ is a symmetric monoidal $\infty$-category via the derived tensor product, see \cite[4.1.7.4]{lurie1}. The notation stems from a localization $\N(\M)\rightarrow \N(\M)[\W^{-1}]$ of $\infty$-categories, where $\N$ denotes the nerve of a category. When $\M$ is combinatorial, the underlying $\infty$-category of $\M$ is also equivalent to $\N(\M_f)[\W^{-1}]$ by \cite[1.3.4.16]{lurie1}.

\item The letter $\k$ shall always denote a commutative ring with global dimension zero. In other words, it is a finite product of fields, i.e.  a commutative ring $\k$ such that it is a product in rings
$
\k=\k_1 \times \cdots \times \k_n$,
where each $\k_i$ is a field, for some $1\leq n < \infty$.
In the literature, such rings are referred to as \emph{commutative semisimple Artinian rings}.

\item Let $\chf$ denote the category of unbounded chain complexes of $\k$-modules (graded homologically). Denote by $\chfo$ the category of non-negative chain complexes over $\k$. Both categories are endowed with a symmetric monoidal structure. The tensor product of two (possibly non-negative) chain complexes $X$ and $Y$ is defined by:
\[
(X\otimes Y)_n=\bigoplus_{i+j=n} X_i\otimes_\k Y_j,
\]
with differential given on homogeneous elements by:
\[
d(x\otimes y)=dx \otimes y + (-1)^{\vert x \vert}x\otimes dy.
\]
Let $\smodk$ be the category of simplicial $\k$-modules. It is endowed with a symmetric monoidal structure with tensor product defined dimensionwise. We denote the tensor of all three categories above simply as $\otimes$. The monoidal unit is denoted $\k$, and is either the chain complex $\k$ concentrated in degree zero, or the simplicial constant module on $\k$.

\item The category $\chf$ is endowed with a model structure in which weak equivalences are quasi-isomor\-phisms, cofibrations are monomorphisms, fibrations are epimorphisms. The category $\chfo$ is endowed with a model structure in which weak equivalence are quasi-isomorphisms. cofibrations are monomorphisms, and fibrations are positive levelwise epimorphisms. The category $\smodk$ is endowed with a model structure in which weak equivalences are weak homotopy equivalences, cofibrations are monomorphisms and fibrations are Kan fibrations. These are all combinatorial and symmetric monoidal model categories, see \cite{hovey} and \cite{monmodSS} for details. Notice that every object is cofibrant and fibrant.

\item Denote by $\NN\colon\smodk\r\chfo$ the normalization functor which is an equivalence of categories. Its inverse is denoted by $\Gamma\colon\chfo\r\smodk$. From \cite{monmodSS}, we obtain two weak monoidal Quillen equivalences:
\begin{equation}\label{eq: dold-kan sym mon}
\begin{tikzcd}[column sep=large]
\chfo\ar[shift left=2]{r}{\Gamma} & \ar[shift left=2]{l}{\mathsf{N}}[swap]{\perp} \smodk,
\end{tikzcd}
\end{equation}
and:
\begin{equation}\label{eq: dold-kan lax mon}
\begin{tikzcd}[column sep=large]
\smodk \ar[shift left=2]{r}{\mathsf{N}} & \ar[shift left=2]{l}{\Gamma}[swap]{\perp}\chfo .
\end{tikzcd}
\end{equation}
In the adjunction (\ref{eq: dold-kan sym mon}), the normalization functor $\NN$ is a right adjoint and is lax symmetric monoidal via the Eilenberg-Zilber map. 
In the adjunction (\ref{eq: dold-kan lax mon}), the normalization functor $\NN$ is a left adjoint and is oplax (but not symmetric) monoidal via the Alexander-Whitney map.

\end{enumerate}



\section{Postnikov Presentations of Model Categories}\label{Sec: terminology}

One of the main tools of model categories is to assume the structure to be \emph{cofibrantly generated} by a pair of sets (see definition in \cite[2.1.17]{hovey}). 
If in addition the category is presentable, we say it is \emph{combinatorial}. 
In this case, cofibrations and acyclic cofibrations are retracts of maps built out of pushouts and transfinite compositions, and we can inductively construct a cofibrant replacement. 
Given a symmetric monoidal model category $\M$ (as in \cite[4.2.6]{hovey}) in which the tensor product interplays well with the generating acyclic cofibrations (see \cite[3.3]{algSS}), then one can guarantee a nice model structure for algebras and modules in $\M$, see \cite[4.1]{algSS}. 
Moreover, given a weak monoidal Quillen equivalence (as in \cite[3.6]{monmodSS}) between symmetric monoidal model categories, the equivalence lifts to a Quillen equivalence between algebras and modules, see \cite[3.12]{monmodSS}. This induces an equivalence of symmetric monoidal $\infty$-categories on the underlying $\infty$-categories of the monoidal model categories, see \cite[2.13]{coalginDK}.

In this section, following \cite{left1}, we recall the dual concept of \emph{fibrantly generated} model categories. We make no requirement of ``cosmallness'' as it is rarely satisfied in practice. Instead when fibrations and acyclic fibrations are retracts of maps built out of pullbacks and towers, we say the model category admits a \emph{Postnikov presentation}. In \cite{left2, left3}, model structures on monoidal categories are lifted to categories of coalgebras and comodules. However there is no guarantee that Postnikov presentations are lifted. We present here the notion of \emph{fibrant-compatible} monoidal structure on a model category which ensures that the homotopy category is endowed with a monoidal structure via a right derived tensor product. We also introduce the notion of \emph{weak opmonoidal Quillen equivalences} which provides a compatibility of the right derived tensor products.

\subsection{Postnikov presentation}

We recall the definition of Postnikov presentations, introduced by Kathryn Hess, which is dual to cellular presentations and appeared in \cite{hess1}, \cite{HScomonad} and \cite{left1}.
We begin with the dual notion of relative cell complex \cite[2.2.9]{hovey}.

\begin{defi}[{\cite[5.12]{hess1}}]
Let $\X$ be a class of morphisms in a category $\C$ closed under pullbacks. Let $\lambda$ be an ordinal. Given a functor $Y\colon\lambda\op\rightarrow \C$ such that for all $\beta <\lambda$, the morphism $Y_{\beta+1}\rightarrow Y_\beta$ fits into the pullback diagram
\[
\begin{tikzcd}
Y_{\beta+1} \ar{r} \ar{d} \pull & X'_{\beta+1} \ar{d} \\
Y_\beta \ar{r} & X_{\beta+1} 
\end{tikzcd}
\]
where $X'_{\beta +1}\rightarrow X_{\beta +1}$ is some morphism in $\X$, and $Y_\beta \rightarrow X_{\beta+1}$ is a morphism in $\C$, and we write
\[
Y_\gamma:=\underset{\beta <\gamma}{\lim} Y_\beta
\]
for any limit ordinal $\gamma <\lambda$. We say that the composition of the tower $Y$:
\[
\llim{\lambda\op} Y_\beta \longrightarrow Y_0, 
\]
if it exists, is a  \emph{$\X$-Postnikov tower}. The class of all $\X$-Postnikov towers  is denoted by $\Post_\X$.
\end{defi}

\begin{prop}[{\cite[2.10]{left1}}]\label{prop: closed post}
If $\C$ is a complete category, the class $\Post_\X$ is the smallest class of morphisms in $\C$ containing $\X$ closed under composition, pullbacks and limits indexed by ordinals.
\end{prop}

\begin{proof}
For a proof, see the dual statements in \cite[2.1.12, 2.1.13]{hovey}.
\end{proof}

\begin{prop}\label{prop: right adjoints preserves cocells}
Let $R\colon\C\rightarrow \D$ be a right adjoint between complete categories. Let $\X$ be a class of morphisms in $\C$. Then we have: $R(\Post_\X)\subseteq \Post_{R(\X)}$.
\end{prop}

\begin{proof}
The claim follows from the fact that right adjoints preserve limits.
\end{proof}

We also recall the dual notion of small object in a category.

\begin{defi}\label{def: cosmall}
Let $\D$ be a subcategory of a complete category $\C$. We say an object $A$ in $\C$ is \emph{cosmall relative to $\D$} if there is a cardinal $\kappa$ such that for all $\kappa$-filtered ordinals $\lambda$ (see \cite[2.1.2]{hovey}) and all $\lambda$-towers $Y:\lambda\rightarrow \D\op$, the induced map of sets:
\[
\collim{\beta<\lambda}\left( \Hom_\C(Y_\beta, A) \right)\longrightarrow\Hom_\C \left( \llim{\beta <\lambda} Y_\beta, A \right),
\]
is a bijection.
We say that $A$ is \emph{cosmall} if it is cosmall relative to $\C$ itself.
\end{defi}

\begin{ex}
The terminal object, if it exists, is always cosmall. In procategories, every object is cosmall. 

\end{ex}

\begin{ex}
A category $\C$ is \emph{copresentable} if its opposite category $\C\op$ is presentable. Therefore in {copresentable} categories, every object is cosmall. However, if $\C$ is presentable then $\C$ is \emph{not} copresentable unless $\C$ is a complete lattice, see \cite[1.64]{Adamek-Rosicky}.
\end{ex}

\begin{ex}\label{ex: cosmall are rare}
As noted after \cite[2.1.18]{hovey}, the only cosmall objects in the category of sets are the empty set and the one-point set.
In practice, objects in a presentable categories are rarely cosmall.
\end{ex}

The dual of the small object argument \cite[2.1.14]{hovey} is stated below. As per our above discussion, in practice, the result is not applicable in most categories of interest.

\begin{prop}[The cosmall object argument]\label{prop: cosmall}
Let $\C$ be a complete category and $\X$ be a set of morphisms in $\C$. If the codomains of maps in $\X$ are cosmall relative to $\Post_\X$, then every morphism $f$ of $\C$ can be factored functorially as:
\[
\begin{tikzcd}
A\ar{rr}{f}\ar{dr}[swap]{\gamma(f)} & & B\\
& C^f, \ar{ur}[swap]{\delta(f)}&
\end{tikzcd}
\]
where $\delta(f)$ is a $\X$-Postnikov tower and $\gamma(f)$ admits the left lifting property with respect to all maps in $\X$.
\end{prop}

\begin{nota}\label{nota: retracts}
Given a class of morphisms $\mathsf{A}$ in $\C$, we denote by $\widehat{\mathsf{A}}$ its closure under formation of retracts.
\end{nota}

\begin{defi}\label{def: cocellular presentation global def}
A \emph{Postnikov presentation $(\X, \Z)$ of a model category $\M$} is a pair of classes of morphisms $\X$ and $\Z$ such that the class of fibrations is $\widehat{\Post_\X}$, the class of acyclic fibrations is $\widehat{\Post_\Z}$, and for any morphism $f\colon X\rightarrow Y$ in $\M$:
\begin{itemize}
\item[(a)] the morphism $f$ factors as:
\[
\begin{tikzcd}
X\ar{rr}{f} \ar{dr}[swap]{i} & & Y\\
& V \ar{ur}[swap]{q}
\end{tikzcd}
\]
where $i$ is a cofibration and $q$ is a $\Z$-Postnikov tower;
\item[(b)] the morphism $f$ factors as:
\[
\begin{tikzcd}
X\ar{rr}{f} \ar{dr}[swap]{j} & & Y\\
& W\ar{ur}[swap]{p}
\end{tikzcd}
\]
where $j$ is an acyclic cofibration and $p$ is a $\X$-Postnikov tower.
\end{itemize}
We say  in this case that the model category \emph{$\M$ is Postnikov presented by $(\X, \Z)$}.
\end{defi}

\begin{rem}\label{rem: cocellular can be trivial with all fib}
We do note require the classes $\X$ or $\Z$ to be sets.
Hence every model category is trivially Postnikov presented by the classes of all fibrations and acyclic fibrations. Although it was noted in \cite[2.13, 2.14]{left1} that this trivial presentation can occasionally be useful (see also \cite[B.4.1]{phd}), we use more interesting subclasses in this paper, see Theorems \ref{thm: cocell presentation of chains} and \ref{thm: cocell presentation of comodules over simply connected}.
\end{rem}

\begin{defi}\label{defi: inductive fibrant replacement}
Let $\M$ be a complete model category that admits a Postnikov presentation $(\X, \Z)$. Given any object $X$ in $\M$, we can provide an \emph{inductive fibrant replacement $FX$} as follows. Let $*$ be the terminal object of $\M$. There is an object $FX$ in $\M$ that factors the trivial map:
\[
\begin{tikzcd}
X\ar{rr} \ar[hook]{dr}[swap]{j} & & *\\
& FX, \ar[two heads]{ur}[swap]{p}
\end{tikzcd}
\]
where $j\colon X\stackrel{\sim}\hookrightarrow FX$ is an acyclic cofibration in $\M$, and $p$ is a $\X$-Postnikov tower.
This means that $FX$ can be defined as iterated maps of pullbacks along $\X$, starting with $(FX)_0=*$.
\end{defi}

In our examples, the above inductive fibrant replacement is a countable tower and its limit is a homotopy limit.

\begin{nota}
We denote by $\operatorname{tow}(\C)$ the category of countable towers in a complete category $\C$:
\[
\begin{tikzcd}
\cdots \ar{r}{f_3} & X(2) \ar{r} {f_2}& X(1) \ar{r}{f_1} & X(0).
\end{tikzcd}
\]
We denote such towers by $\{X(n)\}=(X(n), f_n)_{n\in \bn}$. The limit of the tower is denoted by $\lim_n X(n)$.
\end{nota}

\begin{prop}[{\cite[VI.1.1]{GJ}}]\label{prop: homotopy theory of towers}
Let $\M$ be a model category. Then the category of towers $\operatorname{tow}(\M)$ can be endowed with the Reedy model structure, where a map $\{ X(n)\} \rightarrow \{ Y(n) \}$ is a weak equivalence (respectively a cofibration), if each map $X(n)\rightarrow Y(n)$ is a weak equivalence (respectively a cofibration) in $\M$, for all $n \geq 0$. An object $\{X(n)\}$ is fibrant if and only if $X(0)$ is fibrant and all the maps $X(n+1)\rightarrow X(n)$ in the tower are fibrations in $\M$.
Moreover,  if we denote by $\iota\colon\M\rightarrow \operatorname{tow}(\M)$ the functor induced by the constant diagram, then we obtain a Quillen adjunction $\begin{tikzcd} \iota\colon M\ar[shift left=2]{r}[swap]{\perp} & \operatorname{tow}(\M){\colon \lim_n}. \ar[shift left =2]{l} \end{tikzcd}$
\end{prop}

\subsection{Fibrantly generated model categories}
We recall the notion of fibrantly generated as in \cite{left1}. Our definition of fibrantly generated makes \textbf{no assumption of cosmallness} and is not the dual definition of cofibrantly generated model categories. 

\begin{defi}\label{def: fib gen}
A model category is \emph{fibrantly generated by $(\X, \Z)$} if the cofibrations are precisely the morphisms that have the left lifting property  with respect to $\Z$, and the acyclic cofibrations are precisely the morphisms that have the left lifting property with respect to $\X$. We call $\X$ and $\Z$ the\emph{ generating fibrations} and \emph{generating acyclic fibrations} respectively.
\end{defi}

\begin{rem}
Just as for Remark \ref{rem: cocellular can be trivial with all fib}, since we allow $\X$ and $\Z$ to be classes, any model category is trivially fibrantly generated by its fibrations and acyclic fibrations. In practice, it is useful to have a smaller class to study the model structure. We provide example of non-trivial generating fibrations and acyclic fibrations in Theorem \ref{thm: fib gen of ch}.
\end{rem}

\begin{prop}\label{prop: cocellular imply fib gen}
If a model category is Postnikov presented by a pair of classes $(\X, \Z)$, then it is fibrantly generated by $(\X, \Z)$.
\end{prop}

\begin{proof}
The claim is a direct consequence of the retract argument, see \cite[1.1.9]{hovey}.
\end{proof}

\begin{rem}
The converse of Proposition \ref{prop: cocellular imply fib gen} is true if $\X$ and $\Z$ are sets that permit the cosmall object argument.
However, this rarely happens in context of interest as seen in Example \ref{ex: cosmall are rare}.
\end{rem}

\subsection{Left-induced model categories}
We recall now the terminology of \cite{left2, left3}. Classically, under some assumptions, given a pair of adjoint functors $\begin{tikzcd}[column sep = large]
L:\mathsf{M} \ar[shift left=5pt]{r}[swap]{\perp} & \mathsf{A} \ar[shift left=5pt]{l}:R 
\end{tikzcd}
$ where $\M$ is endowed with a model structure, we can lift a model structure to $\A$ in which weak equivalences and fibrations are created by the right adjoint. See \cite[11.3.2]{hir}.

\begin{defi}
Let $\mathsf{M}$ be a model category and $\mathsf{A}$ be any category, such that there is a pair of adjoint functors:$\begin{tikzcd}[column sep = large]
L:\mathsf{A} \ar[shift left=5pt]{r}[swap]{\perp} & \mathsf{M}:R. \ar[shift left=5pt]{l} 
\end{tikzcd}
$
We say that the left adjoint $L\colon\mathsf{A}\rightarrow \mathsf{M}$ \emph{left-induces} a model structure on $\mathsf{A}$ if the category $\mathsf{A}$ can be endowed with a model structure where a morphism $f$ in $\mathsf{A}$ is defined to be a cofibration (respectively a weak equivalence) if $L(f)$ is a cofibration (respectively a weak equivalence) in $\mathsf{M}$. This model structure on $\mathsf{A}$, if it exists, is called \emph{the left-induced model structure from $\mathsf{M}$}.
\end{defi}

The following result guarantees that left-inducing from a combinatorial model category gives back a combinatorial model category. 

\begin{prop}[{\cite[2.23]{left1},\cite[3.3.4]{left2}}]\label{prop: left-induced preserve combinatorial}
Suppose $\mathsf{A}$ is a model category left-induced by a model category $\M$. Suppose both $\mathsf{A}$ and $\M$ are presentable. If $\M$ is cofibrantly generated by a pair of sets, then $\mathsf{A}$ is cofibrantly generated by a pair of sets.
\end{prop} 

\begin{prop}[{\cite[2.18]{left1}}]\label{prop: left induce preserve fib gen}
Suppose $\mathsf{A}$ is a model category left-induced by a model category $\M$, via an adjunction $
\begin{tikzcd}[column sep=large]
L:\mathsf{A} \ar[shift left=5pt]{r}[swap]{\perp} & \mathsf{M}:R. \ar[shift left=5pt]{l}
\end{tikzcd}
$
If $\M$ is fibrantly generated by $(\X, \Z)$, then $\mathsf{A}$ is fibrantly generated by $(R(\X), R(\Z))$.
\end{prop}

\begin{rem}\label{rem: left-induced does not preserve postnikov}
If $\M$ is Postnikov presented by $(\X, \Z)$, there is no reason to expect that $\mathsf{A}$ is Postnikov presented by $(R(\X), R(\Z))$, as we made no assumption of cosmallness in general.
\end{rem}

\subsection{Weak opmonoidal Quillen pair}  
If $(\C, \otimes)$ is a symmetric monoidal model category, then it endows the associated homotopy category $\mathsf{Ho}(\C)$ with a symmetric monoidal structure endowed with the derived functor $\otimes^\mathbb{L}$, see \cite[4.3.2]{hovey}. Essentially, the conditions ensure that the tensor product of two cofibrant objects is cofibrant, and the tensor preserves weak equivalence in both variable. Therefore we can total left derive the bifunctor $-\otimes-:\C\times \C\rightarrow \C$.
If we denote by $\C_c$ the full subcategory of cofibrant objects in $\C$, then the underlying $\infty$-category $\N(\C_c)\left[ \W^{-1}\right]$ is symmetric monoidal via the derived tensor product $\otimes^\mathbb{L}$, see \cite[4.1.7.4]{lurie1}.

We provide dual conditions such that the homotopy category $\mathsf{Ho}(\C)$ and the underlying $\infty$-category $\N(\C_f)\left[ \W^{-1}\right]$ are endowed with a symmetric monoidal structure. We shall weaken the dual definition of {\cite[4.2.6]{hovey}} as follows. 

\begin{defi}\label{defi: weak comonoidal model cat}
Given a model category $\C$, we say a symmetric monoidal structure $(\M, \otimes, \I)$ is\textit{ fibrant-compatible} if:
\begin{enumerate}[label=\upshape\textbf{(\roman*)}]
\item for any fibrant object $X$ in $\C$, the functors $X\otimes -:\C\rightarrow \C$ and $-\otimes X:\C\rightarrow \C$ preserve fibrant objects and weak equivalences between fibrant objects;
\item\label{item: comonoidal model cat} for any fibrant replacement $\I\rightarrow f\I$ of the unit, the induced morphism $X\cong \I\otimes X\rightarrow f\I\otimes X$ is a weak equivalence, for any fibrant object $X$.
\end{enumerate}
The requirement \ref{item: comonoidal model cat} is automatic if $\I$ is already fibrant
\end{defi}

\begin{ex}
If $\C$ is a symmetric monoidal model category $\C$ then it is in particular\textit{cofibrant-compatible}, in the sense that $\C\op$ is fibrant-compatible. The converse is not true. There may be cofibrant-compatible monoidal structures on model categories that do not form a monoidal model category as $X\otimes-\colon\C\r\C$  needs not to be a left Quillen functor.
\end{ex}

\begin{ex}
Any symmetric monoidal model category $\C$ in which each object is fibrant and cofibrant is also fibrant-compatible. For instance $\C=\smodk, \chfo, \chf$.
\end{ex}

Recall from \cite[2.0.0.1]{lurie1} that to any symmetric monoidal category $\C$, one can define its operator category $\C^\otimes$, such that the nerve $\N(\C^\o)$ is a symmetric monoidal $\infty$-category whose underlying $\infty$-category is $\N(\C)$, see \cite[2.1.2.21]{lurie1}. Given $\C$ a model category with a fibrant-compatible symmetric monoidal structure and  fibrant unit, then $\C_f$ is a symmetric monoidal category. We can apply \cite[4.1.7.6]{lurie1} to obtain a symmetric monoidal structure on the homotopy category of $\C$ (just as in \cite[4.3.2]{hovey}) that is coherent up to higher homotopies.

\begin{prop}[{\cite[4.1.7.6]{lurie1}, \cite[A.7]{tch}}]\label{prop: compatibility of DK in sym mon}
Let $(\C, \otimes ,\I)$ be a model category with a fibrant-compatible symmetric monoidal structure. Suppose that $\bI$ is fibrant.
Then the underlying $\infty$-category $\N(\C_f)[{\W}^{-1}]$ of $\C$ can be given the structure of a symmetric monoidal $\infty$-category via the symmetric monoidal localization of $\N(\C_f^\o)$:
\[
\begin{tikzcd}
\N(\C_f^\o) \ar{r} & \N(\C_f)[{\W}^{-1}]^\o,
\end{tikzcd}
\] 
where ${\W}$ is the class of weak equivalences restricted to fibrant objects in $\C$.
\end{prop}

We can dualize the notion of a weak monoidal Quillen pair from {\cite[3.6]{monmodSS}}.

\begin{defi}\label{defi: weak comonoidal quillen pair}
Let $(\C, \o, \bI)$ and $(\D, \sm, \bJ)$ model categories with fibrant-compatible symmetric monoidal structures.
A \emph{weak opmonoidal Quillen pair} consists of a Quillen adjunction:
\[\begin{tikzcd}
L:(\C, \o, \bI) \ar[shift left=2]{r} & \ar[shift left=2]{l}[swap]{\perp} (\D, \sm, \bJ):R,
\end{tikzcd}
\]
where $R$ is lax monoidal such that the following two conditions hold.
\begin{enumerate}[label=\upshape\textbf{(\roman*)}]
\item\label{enum: weak comonoidal} For all fibrant objects $X$ and $Y$ in $\D$, the lax monoidal map:
\[
\begin{tikzcd}
R(X)\o R(Y) \ar{r} & R(X\Smash Y),
\end{tikzcd}
\]
is a weak equivalence in $\C$.
\item\label{enum; weak comonoidal 2} For some (hence any) fibrant replacement $\lambda\colon \bJ \stackrel{\sim}\longrightarrow f\bJ$ in $\D$, the composite map
\[
\begin{tikzcd}
\I \ar{r} & R(\bJ) \ar{r}{R(\lambda)} & R(f\bJ)
\end{tikzcd}
\]
is a weak equivalence in $\C$.
\end{enumerate}
A weak opmonoidal Quillen pair is a \emph{weak opmonoidal Quillen equivalence} if the underlying Quillen pair is a Quillen equivalence.
\end{defi}

\begin{ex}\label{ex: dold-kan is actually comonoidal}
Given any weak monoidal Quillen pair: \[\begin{tikzcd}
(\C, \o, \bI) \ar[shift left=2]{r}{L} & \ar[shift left=2]{l}{R}[swap]{\perp} (\D, \sm, \bJ),
\end{tikzcd}\] their opposite adjunction: \[\begin{tikzcd}
(\C\op, \o, \bI) \ar[shift right=2]{r}[swap]{L\op} & \ar[shift right=2]{l}{\perp}[swap]{R\op} (\D\op, \sm, \bJ),
\end{tikzcd}\] is a weak opmonoidal Quillen pair.
Because the usual weak monoidal Quillen pair usually requires a monoidal model structure rather than cofibrant-compatible, the converse is not true, but the full assumption of monoidal model structure is not necessary.
\end{ex}

\begin{ex}
The Dold-Kan weak monoidal Quillen equivalences (\ref{eq: dold-kan sym mon}) and (\ref{eq: dold-kan lax mon}) are also weak opmonoidal Quillen equivalences when $k$ is a commutative ring with global dimension zero. This follows from the fact that their unit and counit in the adjunctions are isomorphisms.
\end{ex}

In general, given a weak opmonoidal Quillen equivalence, we do not guarantee that it lifts to a Quillen equivalence on associated categories of coalgebras and comodules. In other words, we do not provide a dual version of \cite[3.12]{monmodSS}. The essential reason being that there is no guarantee in general to obtain a Postnikov presentation on left-induced model structures of coalgebras or comodules such that the limit of the Postnikov towers behaves well with respect to the forgetful functor.

Instead, we show that homotopy coherent (co)algebras and (co)modules are equivalent as $\infty$-categories. In Section \ref{section: dold-kan correspondence}, we provide an example of a weak opmonoidal Quillen pair that lifts to the associated categories of comodules.
One can adapt the proof in \cite{coalginDK} to obtain the following result. 

\begin{thm}[{\cite[2.13]{coalginDK}}]\label{thm: Weak monoidal eq imply strong in infinity}
Let $(\C, \o, \bI)$ and $(\D, \sm, \bJ)$ be model categories with fibrant-compatible symmetric monoidal structures with fibrant units. 
Let $\W_\C$ and $\W_\D$ be the classes of weak equivalence in $\C$ and $\D$ respectively.
Let: \[\begin{tikzcd}
L:(\C, \o, \bI) \ar[shift left=2]{r} & \ar[shift left=2]{l}[swap]{\perp} (\D, \sm, \bJ):R,
\end{tikzcd}\] be a weak opmonoidal Quillen pair.
Then the derived functor of  $R:\D \r \C$ induces a symmetric monoidal functor between the underlying $\infty$-categories:
\[
\begin{tikzcd}
\mathbb{R}:\N (\D_f) \left[\W_\D^{-1}\right] \ar{r} & \N(\C_f) \left[\W_\C^{-1}\right],
\end{tikzcd}
\]
where $\C_f\subseteq \C$ and $\D_f\subseteq \D$ are the full subcategories of fibrant objects.
If $L$ and  $R$ form a weak opmonoidal Quillen equivalence, then $\mathbb{R}$ is a symmetric monoidal equivalence of $\infty$-categories.
\end{thm}

\section{Example: Unbounded Chain Complexes}\label{Sec: example on chains}

We show here that the category $\chf$ of unbounded chain complexes over $\k$ is fibrantly generated by a pair of sets (Theorem \ref{thm: fib gen of ch}) and admits an efficient Postnikov presentation on its model structure (Theorem \ref{thm: cocell presentation of chains}). In particular, this provides new methods to compute homotopy limits in $\chf$ (Remark \ref{rem: homotopy limits in ch}) and thus in the $\infty$-category of modules over the Eilenberg-Mac Lane spectrum $H\k$. The result was proved in \cite{hess1} and \cite{left1} for finitely generated non-negative chain complexes over a field.  

\subsection{The generating fibrations} We present here sets of generating fibrations and acyclic fibrations for $\chf$. A similar result was proved in early unpublished versions of \cite{sore1} in \cite[3.1.11, 3.1.12]{sore0} for non-negative chain complexes over a field.

\begin{defi}\label{defi: D^n and S^n in ch}
Let $V$ be a $\k$-module. Let $n\in \mathbb{Z}$.
Denote the $n$-sphere over $V$ by $S^n(V)$, the chain complex that is $V$ concentrated in degree $n$ and zero elsewhere. 
Denote the $n$-disk over $V$ by $D^n(V)$, the chain complex that is $V$ concentrated in degree $n-1$ and $n$, with differential the identity.
We obtain the epimorphisms $D^n(V)\rightarrow S^n(V)$:
\[
\begin{tikzcd}
D^n(V)\ar{d}  & \cdots & 0\ar[equals]{d} \ar{l} & V\ar{l}\ar{d} & V\ar[equals]{l} \ar[equals]{d}& 0 \ar{l}\ar[equals]{d} & \cdots \ar{l}\\
S^n(V) & \cdots & 0 \ar{l} & 0\ar{l} & V\ar{l} & 0 \ar{l} & \cdots \ar{l}.
\end{tikzcd}
\]
This defines functors:
\[
\begin{tikzcd}
S^n(- ) : \mathsf{Mod}_\k \ar{r} & \chf, & D^n(-): \mathsf{Mod}_\k \ar{r} & \chf.
\end{tikzcd}
\]
The map defined above is natural, i.e. we have a natural transformation $D^n(-)\Rightarrow S^n(-)$, for all $n\in \bz$.
When $V=\k$, we simply write $D^n$ and $S^n$.
Notice that we have a short exact sequence of chain complexes:
\[
\begin{tikzcd}
0\ar{r} & S^{n-1}(V) \ar{r} & D^n(V) \ar{r} & S^n(V) \ar{r} & 0.
\end{tikzcd}
\]
\end{defi}

\begin{defi}\label{def: GENERATING FIB} 
Define $\ii$ and $\jj$ to be the following \emph{sets} of maps in $\chf$:
\[
\begin{tikzcd}
\ii=\{ D^n\longrightarrow S^n\}_{n\in\bz}, &\jj=\{ D^n\longrightarrow 0\}_{n\in \bz}.
\end{tikzcd}
\]
We thicken the sets $\ii$ and $\jj$ to \emph{classes} $\iio$ and $\jj_\oplus$ of morphisms in $\chf$:
\[
\iio :=   \Big\lbrace D^n(V) \longrightarrow S^n(V) \mid V \text{ any $\k$-module} \Big\rbrace_{n\in\bz},
\]
\[
\jj_\oplus :=   \Big\lbrace D^n(V) \longrightarrow 0 \mid V \text{ any $\k$-module} \Big\rbrace_{n\in\bz}.
\]
Clearly, the maps in $\ii$ and $\iio$ are fibrations in $\chf$ and the maps in $\jj$ and $\jj_\oplus$ are acyclic fibrations in $\chf$.
\end{defi}

\begin{rem}
When $\k$ is a field, as every $\k$-module is free, we get:
\[
\ii_\oplus\cong \left\lbrace  \bigoplus_{\lambda} D^n \longrightarrow \bigoplus_{\lambda} S^n\mid \lambda \text{ any ordinal}\right\rbrace_{n\in\bz},
\]
\[
\jj_\oplus \cong \left\lbrace  \bigoplus_{\lambda} D^n \longrightarrow 0\mid \lambda \text{ any ordinal}\right\rbrace_{n\in\bz}.
\]
\end{rem}

\begin{thm}\label{thm: fib gen of ch}
The model category of unbounded chain complexes $\chf$ is fibrantly generated by the pair of sets $(\ii, \jj)$.
\end{thm}

We prove the above theorem with Lemmas \ref{lem: Z in ch} and \ref{lem: X in ch} below.
Given any chain complex $X$, we denote the $n$-boundaries of $X$ by $B_n(X)$ and the $n$-cycles of $X$ by $Z_n(X)$. Recall the following result.

\begin{prop}\label{prop: split in fields}
Let $X$ be a chain complex in $\chf$. Then $X$ is split as a chain complex and we have a non-canonical decomposition: 
\[X_n\cong H_n(X)\oplus B_n(X)\oplus B_{n-1}(X).\]
In particular,  any chain complex $X$ can be decomposed non-canonically as product of disks and spheres:
\[
X \cong \prod_{n\in \mathbb{Z}} S^n(V_n) \oplus D^n(W_n),
\]
where $V_n=H_n(X)$ and $W_n=B_{n-1}(X)$.
\end{prop}

\begin{nota}
Given any class of maps $\mathsf{A}$ in a category $\C$, we denote by $\llp(\mathcal{A})$ the class of maps in $\C$ having the left lifting property with respect to all maps in $\mathsf{A}$.
\end{nota}

\begin{lem}
We have the equalities of classes: $\llp(\ii_\oplus)=\llp(\ii)$ and $\llp(\jj_\oplus)=\llp(\jj)$ in $\chf$.
\end{lem}

\begin{proof}
Since $\ii\subseteq \ii_\oplus$, we get $\llp(\ii_\oplus)  \subseteq \llp(\ii)$.
Suppose now $f$ is in $\llp(\ii)$, let us argue it also belongs in $\llp(\ii_\oplus)$. 
Suppose we have a diagram:
\[
\begin{tikzcd}
X \ar{d}[swap]{f} \ar{r} & D^n(V) \ar{d} \\
Y \ar{r} & S^n(V),
\end{tikzcd}
\]
for some $\k$-module $V$. Since $V$ is projective, there is another $\k$-module $W$ such that $V\oplus W$ is free. Thus $V\oplus W\cong \bigoplus_\lambda \k$, for some basis $\lambda$.
In particular, the choice of the splitting $V\oplus W\rightarrow V$ induces the commutative diagram:
\[
\begin{tikzcd}
X \ar{d}[swap]{f} \ar{r} & D^n(V) \ar{d}\ar[hook]{r} & \displaystyle \ar[dashed,bend right]{l} \bigoplus_{\alpha\in \lambda} D^n_\alpha \ar{d} \ar[hook]{r} &   \ar[dashed,bend right]{l}\displaystyle\prod_{\alpha \in \lambda} D^n_\alpha \ar{r} \ar{d} &D^n_\alpha \ar{d} \\
Y \ar{r} & S^n(V) \ar[hook]{r} & \ar[dashed,bend right]{l}\displaystyle \bigoplus_{\alpha\in \lambda} S^n_\alpha \ar[hook]{r}  & \ar[dashed,bend right]{l} \displaystyle \prod_{\alpha \in \lambda} S^n_\alpha \ar{r}& S^n_\alpha,
\end{tikzcd}
\]
where $D^n_\alpha$ and $S^n_\alpha$ are a copies of $D^n$ and $S^n$.
Since $f$ is in $\llp(\ii)$, we obtain a lift $\ell_\alpha\colon Y \rightarrow D^n_\alpha$, for each $\alpha$. It induces a lift $\ell\colon Y \rightarrow \prod_\alpha D^n_\alpha$ which restricts to $Y\rightarrow D^n(V)$ via the retracts (dashed maps in the diagram).
\end{proof}

\begin{lem}\label{lem: Z in ch}
Maps in the set $\jj$ are the generating acyclic fibrations in $\chf$.
\end{lem}

\begin{proof}
Let $f\colon X\rightarrow Y$ be a map in $\llp(\jj)$, let us show it is a cofibration in $\chf$, i.e. a monomorphism.
Following Proposition \ref{prop: split in fields}, we decompose $X$ as:
\[
X \cong \prod_{n\in \mathbb{Z}} S^n(V_n) \oplus D^n(W_n).
\]
Then the canonical inclusions $S^n(V_n)\hookrightarrow D^{n+1}(V_n)$  induce a monomorphism $\iota$:
\[
\begin{tikzcd}
\iota\colon\displaystyle\prod_{n\in \mathbb{Z}} S^n(V_n) \oplus D^n(W_n) \ar[hook]{r} & \displaystyle \prod_{n\in \mathbb{Z}} D^{n+1}(V_n) \oplus D^n(W_n).
\end{tikzcd}
\]
Since $f$ is in $\llp(\jj)$, then there is a map $\ell$ such that $\iota=\ell\circ f$. Hence $f$ must be a monomorphism.
\end{proof}

\begin{lem}\label{lem: X in ch}
Maps in the set $\ii$ are the generating fibrations in $\chf$.
\end{lem}

\begin{proof}
Notice that $\llp(\ii)\subseteq \llp(\jj)$ as any lift $Y\rightarrow D^n$ in the following commutative diagram induces the dashed lift:
\[
\begin{tikzcd}[column sep=large]
X\ar{r} \ar{d} & D^n\ar{d}\\
Y\ar{r}\ar{ur}\ar[equals]{d} & S^{n}\ar{d}\\
Y \ar{r}\ar[dashed, crossing over]{uur} & 0.
\end{tikzcd}
\]
In particular, Lemma \ref{lem: Z in ch} shows that maps in $\llp(\ii)$ are monomorphisms. 
Let $f\colon X\rightarrow Y$ be the map in $\llp(\ii)$ and let us show it is a quasi-isomorphism. Since $f$ is a monomorphism, there is an induced short exact sequence in $\chf$:
\[
\begin{tikzcd}
0\ar{r} & X\ar{r}{f} & Y \ar{r} & K \ar{r} & 0,
\end{tikzcd}
\]
where $K=\coker(f)$.
It remains to show that $K$ is acyclic.
Notice first that $K$ is defined as the pushout:
\[
\begin{tikzcd}
X \ar{d}[swap]{f}\ar{r} & 0 \ar{d}\\
Y \ar{r} & K, \push
\end{tikzcd}
\]
and so, since $f$ is in $\llp(\ii)$, then $0\rightarrow K$ is in $\llp(\ii)$. 
Following Proposition \ref{prop: split in fields}, decompose $K$ as:
\[
K\cong \prod_{n\in \mathbb{Z}} S^n(V_n) \oplus D^n(W_n),
\]
where $V_n=H_n(K)$.
Then we obtain a map by projection:
\[
\begin{tikzcd}
\displaystyle \prod_{n\in \mathbb{Z}} S^n(V_n) \oplus D^n(W_n) \ar{r} & \displaystyle\prod_{n\in \mathbb{Z}} S^n(V_n),
\end{tikzcd}
\]
that factors through the non-trivial map:
\[
\begin{tikzcd}
\displaystyle\prod_{n\in \bz} D^n(V_n) \ar{r} & \displaystyle\prod_{n\in \bz} S^n(V_n)
\end{tikzcd}
\]
as $0\rightarrow K$ is in $\llp(\ii)=\llp(\iio)$. But this is only possible when $V_n=0$, hence $K$ must be acyclic. Thus $f$ is a quasi-isomorphism.
\end{proof}

\subsection{A Postnikov presentation}
In Definition \ref{def: GENERATING FIB} and Theorem \ref{thm: fib gen of ch}, we introduced sets of generating fibrations and acyclic fibrations for $\chf$. We now show they induce a Postnikov presentation on $\chf$. 

\begin{thm}\label{thm: cocell presentation of chains}
The pair $(\iio, \jj)$ is a Postnikov presentation of the model category of unbounded chain complex $\chf$.
\end{thm}

We shall prove Theorem \ref{thm: cocell presentation of chains} with Lemmas \ref{EASY LEMMA} and \ref{lem: postX in ch} below.

\begin{rem}\label{rem: homotopy limits in ch}
Our arguments in Theorem \ref{thm: cocell presentation of chains} can be generalized to the category $\chf^\D$, the category of functors $\D\rightarrow \chf$, where $\D$ is a small category. Indeed, endow $\chf^\D$ with its injective model structure with levelwise weak equivalences and cofibrations (see for instance \cite[3.4.1]{left2}). Using the same notation as \cite[4.10]{left1}, the pair $(\iio \pitchfork \D, \jj\pitchfork \D)$ is a Postnikov presentation for $\chf^\D$. This provides inductive arguments to compute homotopy limits in $\chf$.
\end{rem}

\begin{rem}\label{rem: few results for the cocells}
We were not able to restrict ourselves to the set $\ii$ and had to consider the class $\iio$. We note here a few basic results. 
\begin{enumerate}[label=\upshape\textbf{(\roman*)}]
\item As $\ii\subseteq \iio$, we get $\Post_\ii\subseteq \Post_\iio$.

\item\label{item: rem for cocell2} The maps $S^n\rightarrow 0$ are in $\Post_\ii$ as they are obtained as pullbacks:
\[
\begin{tikzcd}
S^n\ar{r} \ar{d}\pull & D^{n+1}\ar{d}\\
0 \ar{r} & S^{n+1}.
\end{tikzcd}
\]
Similarly, for any $\k$-module $V$, the maps $S^n(V) \rightarrow 0$ are in $\Post_\iio$.

\item Since $D^n\rightarrow 0$ is the composite $D^n\longrightarrow S^n\longrightarrow 0$,
we see that $\jj\subseteq \Post_\ii$, and thus $\Post_\jj\subseteq \Post_\ii\subseteq \Post_\iio$ by Proposition \ref{prop: closed post}.

\item Although $\iio\nsubseteq \Post_\ii$, we have $\iio\subseteq \widehat{\Post_\ii}$ (see Notation \ref{nota: retracts}). Indeed, for any $\k$-module $V$, any map $D^n(V)\rightarrow S^n(V)$ is the retract of a map $D^n(F)\rightarrow S^n(F)$ where $F$ is a free $\k$-module. Then, for $\lambda$ a basis of $F$, we have the retract in $\chf$:
\[
\begin{tikzcd}
\ds \bigoplus_\lambda D^n\ar{r}\ar{d} & \ds \prod_{\lambda} D^n\ar{r} \ar{d} & \ds \bigoplus_\lambda D^n\ar{d}\\
\ds \bigoplus_\lambda S^n \ar{r} & \ds \prod_\lambda S^n\ar{r} & \ds \bigoplus_\lambda S^n,
\end{tikzcd}
\]
induced by the split short exact sequence in $\k$-modules:
\[
\begin{tikzcd}
0\ar{r} & \ds\bigoplus_\lambda  \k \ar[hook]{r}{\iota} & \ds\prod_\lambda \k \ar[bend right, dashed]{l} \ar{r} & \coker(\iota) \ar{r} & 0,
\end{tikzcd}
\]
where $\iota:\ds\bigoplus_\lambda  \k \hookrightarrow\ds\prod_\lambda \k$ is the natural monomorphism.
\end{enumerate}
\end{rem}

\begin{lem}\label{lem: all fibrants are cocells}
Let $X$ be any chain complex over $\k$. Then the trivial map $X\rightarrow 0$ is a $\iio$-Postnikov tower. If $X$ is acyclic, the trivial map is a $\jj_\oplus$-Postnikov tower.
\end{lem}

\begin{proof}
Follows from Proposition \ref{prop: split in fields}, \ref{item: rem for cocell2} of Remark \ref{rem: few results for the cocells}, and Proposition \ref{prop: closed post}.
\end{proof}

\begin{lem}\label{EASY LEMMA}
Every acyclic fibration in $\chf$ is a retract of a $\jj$-Postnikov tower. Every map in $\chf$ factors as a cofibration followed by a $\jj$-Postnikov tower.
\end{lem}

\begin{proof}
We provide two proofs by presenting two different factorizations. The first one has the advantage to be functorial but harder to compute. The second is not functorial but is easier to compute.
Let us do the first possible factorization. By Theorem \ref{thm: fib gen of ch}, the set $\jj$ of maps in $\chf$ is the set of generating acyclic fibrations. Their codomain is the terminal object in $\chf$ and is thus cosmall. We can then apply the cosmall object argument (Proposition \ref{prop: cosmall}) to obtain the desired factorization.

For the second possible factorization, start with any morphism $f\colon X \rightarrow Y$ in $\chf$. 
Choose a decomposition of $X$ by using Proposition \ref{prop: split in fields}:
$X \cong  \prod_{n\in \bz} S^n(V_n) \oplus D^n(W_n)$,
for some collection of $\k$-modules $V_n$ and $W_n$. 
The inclusions $S^n(V_n)\hookrightarrow D^{n+1}(V_n)$ define then a monomorphism in $\chf$:
\[
\begin{tikzcd}
X\cong \ds \prod_{n\in \bz} S^n(V_n) \oplus D^n(W_n)\ar[hook]{r} & \ds \prod_{n\in \bz} D^{n+1}(V_n) \oplus D^n(W_n).
\end{tikzcd}
\]
Since $V_n$ and $W_n$ are projective $\k$-modules, they can be embedded into free $\k$-modules, say $F_n$ and $G_n$, with basis $\lambda_n$ and $\gamma_n$.
Then we obtain the following monomorphisms in $\chf$:
\[
\begin{tikzcd}
\ds \prod_{n\in \bz} D^{n+1}(V_n) \oplus D^n(W_n) \ar[hook]{r} & \ds\prod_{n\in \bz} \bigoplus_{\lambda_n} D^{n+1} \oplus \bigoplus_{\gamma_n} D^n.
\end{tikzcd}
\]
Thus, we get the monomorphism in $\chf$:
$\begin{tikzcd}
X\ar[hook]{r} & \ds \prod_{n\in \bz} \prod_{\lambda_n, \gamma_n} D^{n+1} \oplus D^n.
\end{tikzcd}$
Denote  by$Z$ the acyclic chain complex  $\prod_{n\in \bz} \prod_{\lambda_n, \gamma_n} D^{n+1} \oplus D^n$. 
We obtain the desired second factorization in $\chf$: $\begin{tikzcd}
X\ar[hook]{r}{\iota \oplus f} & Z\oplus Y \ar{r}{q} & Y,
\end{tikzcd} $
where the map $q$ is the projection onto $Y$, which is indeed a $\jj $-Postnikov tower by Proposition \ref{prop: closed post}.
\end{proof}

\begin{lem}\label{lem: postX in ch}
Every fibration in $\chf$ is a retract of a $\iio$-Postnikov tower and every map in $\chf$ factors as an acyclic cofibration followed by a $\iio$-Postnikov tower.
\end{lem}

We prove the above lemma at the end of this section.

\begin{rem}
Unlike Lemma \ref{EASY LEMMA}, we cannot use the cosmall object argument in order to prove Lemma \ref{lem: postX in ch}.
Indeed, as noted in \cite{sore0}, the codomains $S^n$ of maps in the set $\ii$ are not cosmall relative to $\Post_\ii$. Indeed, let $Y_k=S^n$ for all $k\geq 0$ and $Y_{k+1}\rightarrow Y_k$ be the zero maps. Let $Y=\llim{k\geq 0}Y_k$ be the limit in $\chf$. The set map:
\[
\collim{k\geq 0}\left(\Hom_\chf(Y_k, S^n)\right) \longrightarrow \Hom_\chf(Y, S^n)
\]
is not a bijection. Indeed, the map is equivalent to the map:
\[
\bigoplus_{k\geq 0} \k \longrightarrow \left(\prod_{k\geq 0} \k\right)^*,
\]
which is never a bijection. 
A similar argument can be applied to show that the codomains $S^n(V)$ of the maps in the class $\iio$ are not cosmall relative to $\Post_\iio$, for any $\k$-module $V$.
\end{rem}

\begin{lem}\label{lem: pullback homology of X}
Let $X$ be a chain complex in $\chf$. Let $V$ be a $\k$-module. Let $n\in\bz$.
Given a surjective linear map $f_n\colon X_n\rightarrow V$ non-trivial only on the $n$th-homology in a splitting $X_n\cong H_n(X)\oplus B_n(X)\oplus B_{n-1}(X)$, there is a  map of chain complexes $f\colon X\rightarrow S^n(V)$, and the pullback chain complex $P$ in the following diagram:
\[
\begin{tikzcd}
P\pull \ar{r} \ar{d} & D^n(V)\ar{d}\\
X\ar{r} & S^n(V),
\end{tikzcd}
\]
has homology:
\[
H_i(P)\cong\left\lbrace \begin{tabular}{ll}
$\ker\left(H_n(f)\right)$ & $i=n$,\\
$H_i(X)$ & $i\neq n$,
\end{tabular} \right.
\]
and we have $P_i=X_i$ for $i\neq n-1$ and $P_{n-1}=X_{n-1}\oplus V$.
\end{lem}

\begin{proof}
By construction, since pullbacks in $\chf$ are taken levelwise, for $i\neq n, n-1$, we have the pullbacks of $\k$-modules:
\[
\begin{tikzcd}
P_{n-1}\pull\ar{r} \ar{d} & V\ar{d}  & P_n\pull\ar{r}\ar{d} & V\ar[equals]{d} & P_i\pull\ar{r}\ar{d} & 0\,\ar[equals]{d}\\
X_{n-1} \ar{r} & 0, & X_n\ar{r}{f_n} & V, & X_i \ar{r} & 0.
\end{tikzcd}
\]
Thus $P_{n-1}\cong X_{n-1}\oplus V$ and $P_{i}=X_i$ for any $i\neq {n-1}$. The differential $P_n\rightarrow P_{n-1}$ is the linear map $
\begin{tikzcd}
X_n\ar{r}{d_n\oplus f} & X_{n-1}\oplus V,
\end{tikzcd}$
and the differential $P_{n-1}\rightarrow P_{n-2}$ is the linear map:
\[
\begin{tikzcd}
X_{n-1}\oplus V \ar{r} & X_{n-1}\ar{r}{d_{n-1}} & X_{n-2},
\end{tikzcd}
\]
where the unlabeled map is the natural projection. All the differentials $P_i\rightarrow P_{i-1}$ for $i\neq n, n-1$ are the differentials $X_i\rightarrow X_{i-1}$ of the chain complex $X$. Clearly, we get $H_i(P)=H_i(X)$ for $i\neq n, n-1$. For $i=n-1$, by Proposition \ref{prop: split in fields}, we can choose a decomposition:
\[
X_n\cong H_n(X)\oplus B_{n-1}(X)\oplus B_n(X).
\]
The differential $d_n\colon X_n\rightarrow X_{n-1}$ sends the factor $B_{n-1}(X)$ in $X_n$ to itself, and the factor $H_n(X)\oplus B_n(X)$ to zero. By definition, the map $f_n:X_n\rightarrow V$ sends the factor $H_n(X)$ in $X_n$ to the image of $f_n$, which is $V$ since $f_n$ is surjective, and the factor $B_{n-1}(X)\oplus B_n(X)$ to zero.
Thus the image of the differential $P_n\longrightarrow P_{n-1}$, is precisely $B_{n-1}(X)\oplus V$.
Therefore, we obtain:
\begin{eqnarray*}
H_{n-1}(P) & = & \frac{\ker(P_{n-1}\rightarrow P_{n-2})}{\im(P_n\rightarrow P_{n-1})}\\
& \cong & \frac{Z_{n-1}(X)\oplus V}{B_{n-1}(X)\oplus V}\\
& \cong & \frac{Z_{n-1}(X)}{B_{n-1}(X)}\\
& = & H_{n-1}(X).
\end{eqnarray*}
For $i=n$, notice that the $n$-boundaries of $P$ are precisely the $n$-boundaries of $X$, the $n$-cycles of $P$ are the $n$-cycles $x$ in $X$ such that $f_n(x)=0$. 
Since $f_n\colon X_n\rightarrow V$ is entirely defined on the copy $H_n(X)$ in $X_n$, 
we get 
that $H_n(P)\cong \ker(H_n(f)).$
\end{proof}

\begin{lem}\label{lem: technical factorization in ch}
Let $j\colon X\rightarrow Y$ be a monomorphism in $\chf$ (i.e. a cofibration), such that it induces a monomorphism in homology in each degree. Let $n\in \bz$.
Then the map $j$ factors in $\chf$ as:
\[
\begin{tikzcd}
X\ar[hook]{rr}{j}\ar[hook]{dr}[swap]{j(n)} & & Y\\
& Y(n) \ar{ur}[swap]{p(n)} &
\end{tikzcd}
\]
where $Y(n)$ is a chain complex built with the following properties.
\begin{itemize}
\item The chain map $p(n)\colon Y(n)\rightarrow Y$ is a $\iio$-Postnikov tower.

\item The chain map $j(n)\colon X\rightarrow Y(n)$ is a monomorphism (i.e. a cofibration in $\chf$).

\item The $\k$-module $(Y(n))_i$ differs from $Y_i$ only in degree $i=n-1$.

\item In degrees $i\neq n$ in homology, we have $H_i(Y(n))\cong H_i(Y)$ and the maps:
\[
H_i(j(n))\colon H_i(X)\longrightarrow H_i(Y(n))\cong H_i(Y),
\]
are precisely the maps $H_i(j)\colon  H_i(X)\rightarrow H_i(Y)$.
In particular, the maps $H_i(j(n))$ are monomorphisms. Moreover, if the maps $H_i(j)$ are isomorphisms, then so are the maps $H_i(j(n))$.
\item In degree $n$ in homology, the map
$
H_n(j(n))\colon H_n(X)\stackrel{\cong}\longrightarrow H_n(Y(n))
$
is an isomorphism.
\end{itemize}
\end{lem}

\begin{proof} We construct below the chain complex $Y(n)$ explicitly using Lemma \ref{lem: pullback homology of X}.
By Proposition \ref{prop: split in fields}, we can decompose $Y_n$ as:
\[
Y_n\cong H_n(Y) \oplus \overline{Y_n}\cong \im\Big( H_n(j)\Big) \oplus \coker\Big( H_n(j) \Big) \oplus \overline{Y_n},
\]
where $\overline{Y_n}$ is the direct sum of the copies of the  boundaries.
Denote the $\k$-module $\coker( H_n(j))$ as $V$ and define the linear map $f_n\colon Y_n\rightarrow V$ to be the natural projection. In particular, the map $f_n$ sends $n$-boundaries of $Y$ to zero. This defines a chain map:
$f\colon Y\longrightarrow S^n(V)$.
Notice that since $j\colon X\rightarrow Y$ is a monomorphism, we get $j(\overline{X_n})\subseteq \overline{Y_n}$, and so, by construction of $f$, we get that the composite: $
\begin{tikzcd}
X\ar{r}{j} & Y \ar{r}{f} & S^n(V),
\end{tikzcd}
$
is the zero chain map. 
We obtain $Y(n)$ as the following pullback in $\chf$, with a chain map $j(n)$ induced by the universality of pullbacks:
\[
\begin{tikzcd}
X\ar[dashed]{dr}{j(n)}[swap]{\exists !} \ar[bend left]{drr}{0} \ar[bend right=70pt, hook]{ddr}[swap]{j} & & \\
& Y(n)\pull\ar{r}\ar{d}[swap]{p(n)} & D^n(V) \ar{d}\\
& Y\ar{r}[swap]{f} & S^n(V).
\end{tikzcd}
\]
By construction, the induced chain map $p(n)\colon Y(n)\rightarrow Y$ is in $\Post_\iio$. From the commutativity of the diagram:
\[
\begin{tikzcd}[row sep=large, column sep=large]
X \ar{r}{j(n)}\ar[hook]{dr}[swap]{j} & Y(n)\ar{d}{p(n)} \\
& Y,
\end{tikzcd}
\]
since $j$ is a monomorphism, so is $j(n)$. Since $H_i(j)$ is a monomorphism for $i\in\bz$, then so is $H_i(j(n))$.
By Lemma \ref{lem: pullback homology of X}, we get $H_i(Y(n))\cong H_i(Y)$ for all $i\neq n$. For $i=n$, we get:
\[
\begin{tikzcd}
H_n(Y(n))\cong \ker \Big( H_n(f) \Big) \cong H_n(X),
\end{tikzcd}
\]
as we have the short exact sequence of $\k$-modules:
\[
\begin{tikzcd}[column sep= large]
0 \ar{r} & H_n(X) \ar{r}{H_n(j)} & H_n(Y) \ar{r}{H_n(f)} & V \ar{r} & 0,
\end{tikzcd}
\]
since $V=\coker( H_n(j) )$. Thus $H_n(j(n))$ is an isomorphism as desired.
\end{proof}

\begin{proof}[Proof of Lemma \ref{lem: postX in ch}]
The first statement follows from the second using the retract argument.
Given a chain map $f\colon X\rightarrow Y$, we build below a chain complex $\widetilde{W}$ as a tower in $\chf$ using Lemma \ref{lem: technical factorization in ch} repeatedly so that $f$ factors as:
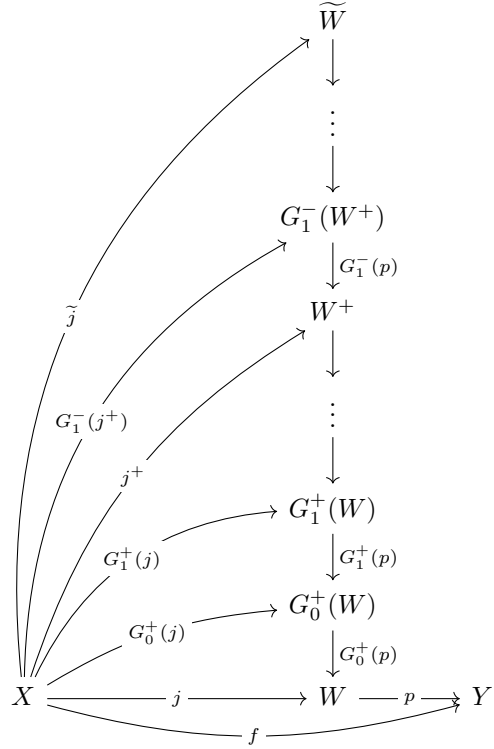
\[
\begin{tikzcd}
&[+60pt] \widetilde{W}\ar{d} & \\
& \vdots\ar{d} & \\
& W^+(-1)\ar{d}{p^-(-1)} & \\
& W^+\ar{d} & \\
& \vdots\ar{d} & \\
& W^+(1)\ar{d}{p^+(1)} & \\
& W(0)\ar{d}{p^+(0)} & \\
X \ar[bend right=15pt]{rr}[description]{f}\ar{r}[description]{j}\ar[bend left=13pt]{ur}[description]{j^+(0)}\ar[bend left]{uur}[description]{j^+(1)}\ar[bend left=20pt]{uuuur}[description]{j^+}\ar[bend left]{uuuuur}[description]{j^-(-1)}\ar[bend left]{uuuuuuur}[description]{\widetilde{j}} & W \ar{r}[description]{p} & Y
\end{tikzcd}
\]
where $\widetilde{j}$ is a monomorphism and a quasi-isomorphism, and all the vertical maps and $p$ are in $\Post_\iio$. The composition of all the vertical maps and $p$ is a chain map $\widetilde{W}\rightarrow Y$ which is in $\Post_\iio$, by Proposition \ref{prop: closed post}.

We first start by noticing the following factorization:
\[
\begin{tikzcd}
X\ar{rr}{f}\ar[hook]{dr}[swap]{j} & & Y,\\
& X\oplus Y \ar{ur}[swap]{p}
\end{tikzcd}
\]
induced by the following pullback in $\chf$:
\[
\begin{tikzcd}
X\ar[equals, bend left]{drr} \ar[bend right]{ddr}[swap]{f} \ar[hook]{dr}{j} & & \\
& X \oplus Y \pull \ar{d}{p} \ar{r} & X \ar{d}\\
& Y \ar{r} & 0.
\end{tikzcd}
\]
The map $p$ is in $\Post_\iio$ by Lemma \ref{lem: all fibrants are cocells} and Proposition \ref{prop: closed post}. By commutativity of the upper triangle, we see that the monomorphism $j$ induces a monomorphism in homology. We write $W=X\oplus Y$.

The second step is to replace the map $j\colon X\rightarrow W$ by a chain map $j^+\colon X\rightarrow W^+$ that remains a cofibration, a monomorphism in homology in negative degrees, and an isomorphism in homology in non-negative degrees. We construct $W^+$ as the limit $\llim{n\geq 0}(W^+(n))$ in $\chf$ of the tower of maps:
\[
\begin{tikzcd}[row sep=large]
\cdots \ar{r} & W^+(2)\ar{r}{p^+(2)} & W^+(1)\ar{r}{p^+(1)} & W^+(0) \ar{r}{p^+(0)} & W,
\end{tikzcd}
\]
where each $p^+(n)$ is in $\Post_\iio$. The map $j^+\colon X\rightarrow W^+$ is induced by the monomorphisms $j^+(n)\colon X\rightarrow W^+(n)$ which are compatible with the tower:
\[
\begin{tikzcd}[row sep=large]
&[+20pt] W^+(n)\ar{d}{p^+(n)}\\
X \ar{ur}{j^+(n)} \ar{r}[swap]{j^+(n-1)} & W^+(n-1),
\end{tikzcd}
\]
and $j^+(n)$ induces an isomorphism in homology in degrees $i$, for $0\leq i \leq n$, and a monomorphism otherwise.
We construct the chain complexes $W^+(n)$ of the tower inductively as follows. 
\begin{itemize}
\item For the initial step, apply Lemma \ref{lem: technical factorization in ch} to the monomorphism $j\colon X\rightarrow W$, for $n=0$. Write $W^+(0)=W(0)$. The cofibration $j^+(0)$ defined as the chain map
\[
j(0)\colon X\longrightarrow W(0)=W^+(0)
\]
is an isomorphism in homology in degree $0$, and a monomorphism in other degrees.
The chain map $p^+(0)$ defined as the map:
\[
p(0)\colon W^+(0)=W(0)\longrightarrow W,
\]
is a $\iio$-Postnikov tower.

\item For the inductive step, suppose, for a fixed integer $n\geq 0$, the chain complex $W^+(n)$ is defined, together with a cofibration $j^+(n)\colon X\rightarrow W^+(n)$ inducing an isomorphism in homology for degrees $i$, where $0\leq i\leq n$, and a monomorphism in homology for other degrees. 
Apply Lemma \ref{lem: technical factorization in ch} to the monomorphism $j^+(n)$ for the degree $n+1$. Write $W^+(n+1):=(W^+(n))(n+1)$.
The cofibration $j^+(n+1)$ defined as the chain map:
\[
(j^+(n))(n+1)\colon X\longrightarrow W^+(n+1),
\]
is an isomorphism in homology in degrees $i$ where $0\leq i\leq n+1$, and a monomorphism in other degrees. We obtain a $\iio$-Postnikov tower  $p^+(n+1)\colon W^+(n+1)\rightarrow W^+(n)$ 
such that the following diagram commutes:
\[
\begin{tikzcd}
X\ar[hook]{rr}{j^+(n)}\ar[hook]{dr}[swap]{j^+(n+1)} & & W^+(n)\\
 & W^+(n+1).\ar{ur}[swap]{p^+(n+1)}
\end{tikzcd}
\]
\end{itemize}
The induced map $j^+\colon X\rightarrow W^+$ is a monomorphism of chain complexes. Indeed, for any fixed $i\in \bz$, we have:
\[
{\left(W^+(i+1)\right)}_i\cong{\left(W^+(i+2)\right)}_i\cong{\left(W^+(i+3)\right)}_i\cong\cdots.
\]
Thus $\left(W^+\right)_i\cong{\left(W^+(i+1)\right)}_i$.
Therefore the linear map $(j^+)_i\colon X_i\rightarrow (W^+)_i$ is the linear map:
\[
{\left(j^+(n+1)\right)}_i\colon X_i \longrightarrow {\left(W^+(i+1)\right)}_i,
\]
which is a monomorphism. Similarly, we get:
$
H_i(W^+)\cong H_i(W^+(i+1)),
$
for all $i\in\bz$, and so $j^+$ is a monomorphism in negative degrees in homology, and an isomorphism in homology in non-negative degrees.

The last step is to replace the map $j^+\colon  X\rightarrow W^+$ by the desired chain map $\wt{j}:X\rightarrow \wt{W}$ that is an acyclic cofibration. We construct $\wt{W}$ similarly as $W^+$ (inductively applying Lemma \ref{lem: technical factorization in ch}) but in negative degrees.
We build $\wt{W}$ as the limit $\llim{n\geq 0}(W^-(-n))$ in $\chf$ of the tower of maps:
\[
\begin{tikzcd}[row sep=large]
\cdots \ar{r} & W^+(-2)\ar{r}{p^-(-2)} & W^-(-1) \ar{r}{p^-(-1)} & W^-(0)=W^+,
\end{tikzcd}
\]
where each $p^-(n)$ is in $\Post_\iio$. The map $\wt{j}\colon X\rightarrow \wt{W}$ is induced by the monomorphisms $j^-(-n)\colon X\rightarrow W^-(-n)$ which are compatible with the tower:
\[
\begin{tikzcd}[row sep=large]
&[+20pt] W^-(-n)\ar{d}{p^-(-n)}\\
X \ar{ur}{j^-(-n)} \ar{r}[swap]{j^-(-(n-1))} & W^-(-(n-1)),
\end{tikzcd}
\]
and $j^-(-n)$ induces an isomorphism in homology in degrees $i$, for $i\geq -n$, and a monomorphism otherwise. Similarly as the positive case, the map $\wt{j}\colon X\rightarrow \wt{W}$ can be shown to be a monomorphism and quasi-isomorphism, hence an acyclic cofibration, as desired.
\end{proof}

\section{Example: Simplicial and Connective Comodules}\label{Sec: example on comodules}

We introduce here, in Theorem \ref{thm: cocell presentation of comodules over simply connected}, a Postnikov presentation for simplicial and connective differential graded comodules, over simply connected coalgebras (Definition \ref{defi: simply connected}). This provides us with an efficient inductive fibrant replacement as a Postnikov tower (Corollary \ref{SUPER IMPORTANT COR}). We begin to adapt our results on $\chf$ in previous section to $\chfo$ and $\smodk$.

\begin{defi}\label{defi: generating dold kan fibrations}
Denote by $\tau_{\geq 0} \colon\chf \rightarrow \chfo$ the $0$-th (homological) truncation (see \cite[1.2.7]{weibel}). From the sets and classes of Definition \ref{def: GENERATING FIB}, we denote their image under the truncation by:
\[
\begin{tikzcd}[column sep= small]
\iii=\{ D^n\longrightarrow S^n\}_{n\geq 1}\cup \{0\rightarrow S^0\}, &\jjj=\{ D^n\longrightarrow 0\}_{n\geq 1},
\end{tikzcd}
\]
and:
\begin{eqnarray*}
\iiio &:= &  \Big\lbrace D^n(V) \longrightarrow S^n(V) \mid V \text{ any $\k$-module} \Big\rbrace_{n\geq 1}\\
& &  \bigcup  \Big\lbrace 0 \longrightarrow S^0(V) \mid V \text{ any $\k$-module} \Big\rbrace.
\end{eqnarray*}
We also obtain sets and classes in $\smodk$ via the equivalence $\Gamma\colon\chfo\rightarrow \smodk$. For $n\geq 0$, we denote by $\emm:=\Gamma(S^n(V))$ which corresponds precisely to the simplicial model of the Eilenberg-Mac Lane space of $V$ in degree $n$.
For $n\geq 1$, we denote by $\emmm:=\Gamma(D^n(V))$ which corresponds to the based path of $\emm$.  Define:
\[
\iid=\{ PK(\k, n) \longrightarrow K(\k, n)\}_{n\geq 1}\cup \{0\rightarrow K(\k, 0)\}, 
\]
\[
\jjd=\{ PK(\k, n)\longrightarrow 0\}_{n\geq 1},
\]
and:
\begin{eqnarray*}
\iido &:= &  \Big\lbrace PK(V,n) \longrightarrow K(V,n) \mid V \text{ any $\k$-module} \Big\rbrace_{n\geq 1}\\
& &  \bigcup  \Big\lbrace 0 \longrightarrow K(V, 0) \mid V \text{ any $\k$-module} \Big\rbrace.
\end{eqnarray*}
\end{defi}

The Quillen adjunction: 
$\begin{tikzcd}
\chfo \ar[shift left=2, hook]{r} & \ar[shift left=2]{l}[swap]{\perp} \chf:{\tau_{\geq 0}}
\end{tikzcd}$ shows that the model structure on $\chfo$ is in fact left-induced from $\chf$. From Theorem \ref{thm: fib gen of ch} and Proposition \ref{prop: left induce preserve fib gen}, we obtain that $\chfo$ and $\smodk$ are fibrantly generated. As pointed out in Remark \ref{rem: left-induced does not preserve postnikov}, it is not automatic that we obtain Postnikov presentations for $\chfo$ and $\smodk$ from the one in Theorem \ref{thm: cocell presentation of chains}. However, we claim we can adapt the arguments in previous section to obtain the following result. Alternatively, one can apply Theorem \ref{thm: cocell presentation of comodules over simply connected} below with $C=\k$. 

\begin{thm}
Let $\k$ be a commutative ring with global dimension zero.
\begin{enumerate}[label=\upshape\textbf{(\roman*)}]

\item The model category of non-negative chain complexes $\chfo$ is fibrantly generated by the pair of sets $(\iii, \jjj)$. The pair $(\iiio, \jjj)$ is a Postnikov presentation of $\chfo$.

\item The model category of simplicial $\k$-modules $\smodk$ is fibrantly generated by the pair of sets $(\iid, \jjd)$. The pair $(\iido, \jjd)$ is a Postnikov presentation of $\smodk$.
\end{enumerate}
\end{thm}

\subsection{Left-induced model structures on comodules}

We present here the model structures for como\-dules in $\smodk$ and $\chfo$, due to \cite{left2, left3}. We first introduce general definitions.

\begin{defi}\label{defi: coalgebra}
Let $(\C, \otimes, \bI)$ be a symmetric monoidal category.
A \emph{coalgebra} $(C, \Delta, \varepsilon)$ in $\C$ consists of an object $C$ in $\C$ together with a coassociative comultiplication $\Delta \colon C\rightarrow C\otimes C$, such that the following diagram commutes: 
\[
\begin{tikzcd}[column sep =large]
C\ar{r}{\Delta} \ar{d}[swap]{\Delta} &C\otimes C\ar{d}{\textup{id}_C\otimes \Delta}\\
C\otimes C \ar{r}{\Delta\otimes \textup{id}_C} & C\otimes C\otimes C,
\end{tikzcd}
\]
and admits a counit morphism $\varepsilon\colon C\rightarrow \mathbb{I}$ such that we have the following commutative diagram: 
\[
\begin{tikzcd}
C\otimes C \ar{r}{\textup{id}_C\otimes \varepsilon} & C\otimes \mathbb{I} \cong C \cong \mathbb{I}\otimes C & C\otimes C \ar{l}[swap]{\varepsilon \otimes \textup{id}_C}\\
& C.\ar[equals]{u}\ar[bend left]{ul}{\Delta} \ar[bend right]{ur}[swap]{\Delta} &
\end{tikzcd}
\]
The coalgebra is \emph{cocommutative} if the following diagram commutes: 
\[
\begin{tikzcd}
C\otimes C \ar{rr}{\tau} && C\otimes C\\
& C,\ar{ul}{\Delta}\ar{ur}[swap]{\Delta} &
\end{tikzcd}
\]
where $\tau$ is the twist isomorphism from the symmetric monoidal structure of $\C$. A morphism of coalgebras $f\colon (C,\Delta, \varepsilon)\rightarrow (C',\Delta',\varepsilon')$ is a morphism $f\colon C\rightarrow C'$ in $\C$ such that the following diagrams commute: 
\[
\begin{tikzcd}
C\ar{r}{f} \ar{d}[swap]{\Delta} & C'\ar{d}{\Delta'} & C\ar{r}{f}\ar{dr}[swap]{\varepsilon} & C'\ar{d}{\varepsilon'}\\
C\otimes C \ar{r}{f\otimes f} & C'\otimes C', & & \I.
\end{tikzcd}
\]
\end{defi}

\begin{defi}\label{defi: comodules in ordinary}
Let $(\C, \otimes, \I)$ be symmetric monoidal category.
Let $(C, \Delta, \varepsilon)$ be a coalgebra in $\C$. A \emph{right comodule $(X, \rho)$ over $C$}, or a \emph{right $C$-comodule}, is an object $X$ in $\C$ together with a coassociative and counital right coaction morphism $\rho\colon X\rightarrow X\o C$ in $\C$, i.e., the following diagram commutes:
\[
\begin{tikzcd}
X\ar{r}{\rho}\ar{d}[swap]{\rho} & X\o C \ar{d}{\rho \o \id_C}  & & X \ar{r}{\rho}\ar[equals]{ddr} & X\o C \ar{d}{ \id_X \o \varepsilon}\\
X\o C \ar{r}[swap]{ \id_X\o \Delta} &X\o C\o C, & & & X\o \bI  \ar{d}{\cong} \\[-5pt]
  & & &  & X. 
\end{tikzcd}
\]
The category of right $C$-comodules in $\C$ is denoted by $\comod_C(\C)$. Similarly, we can define the category of left $C$-comodules where objects are endowed with a left coassociative counital coaction $X\rightarrow C\o X$ and we denote the category by ${}_{C}\mathsf{CoMod}(\C)$.
\end{defi}

\begin{rem}
If $C$ is a cocommutative comonoid in $\C$ the categories of left and right comodules over $C$ are naturally isomorphic: $
{}_{C}\mathsf{CoMod}(\C)\cong \mathsf{CoMod}_C(\C)$.
In this case, we omit to mention if the coaction is left or right.
\end{rem}



\begin{defi}
For any object $X$ in $\C$, we say that $X\otimes C$ is the \emph{cofree right $C$-comodule generated by $X$}. Similarly, we can define the \emph{cofree left $C$-comodule generated by $X$ as $C\otimes X$}.
\end{defi}

\begin{prop}[{\cite[2.2]{connectivecomod}}]
Let $\M$ be either the monoidal model category $\smodk$ or $\chfo$.
Let $C$ be a coalgebra in $\M$.
Then the category of right $C$-comodules in $\M$ admits a combinatorial model category left-induced from the forgetful-cofree adjunction:
\[
\begin{tikzcd}[column sep= large]
\comod_C(\M) \ar[shift left=2]{r}{U}[swap]{\perp} & \M. \ar[shift left=2]{l}{-\o C}
\end{tikzcd}
\]
In particular, $U$ preserves and reflects cofibrations and weak equivalences. The model structure is combinatorial and simplicial. Every object is cofibrant.
\end{prop}

\begin{cor}
Let $\k$ be a commutative ring with global dimension zero.
\begin{enumerate}[label=\upshape\textbf{(\roman*)}]
\item Let $C$ be a coalgebra in $\chfo$. The model category of  right $C$-comodules $\comod_C(\chfo)$ is fibrantly generated by the pair of sets $(\iii \otimes C, \jjj\otimes C)$.
\item Let $C$ be a coalgebra in $\smodk$. The model category of right $C$-comodules $\comod_C(\smodk)$ is fibrantly generated by the pair of sets $(\iid\otimes C, \jjd\otimes C)$.
\end{enumerate}
\end{cor}

\subsection{Postnikov presentations} We now introduce a Postnikov presentation for comodules in $\smodk$ and $\chfo$. We follow the approach of \cite{hess1} in which the finitely generated case was treated. See also \cite{joseph}. We first need to restrict the coalgebras considered.

\begin{defi}\label{defi: simply connected}
Let $R$ be any commutative ring. 
A chain complex $X$ over $R$ is \emph{simply connected} if: $X_0=R$, $X_1=0$ and $X_i=0$ for all $i<0$. A simplicial $R$-module $X$ is \emph{simply connected} if $X_0=R$ and there are no non-degenerate $1$-simplices.
\end{defi}

Over a commutative ring with global dimension zero $\k$, every simply connected simplicial modules or chain complexes is weakly equivalent to a simply connected object in the usual sense.

\begin{rem}
A simplicial $R$-module $X$ is simply connected if and only if $\mathsf{N}(X)$ is a simply connected chain complex. 
A chain complex $X$ over $R$ is simply connected if and only if $\Gamma(X)$ is a simply connected simplicial module.
\end{rem}
For any chain complex $C$ and any $\k$-module $V$, we see that the $i$-th term of the chain complex $S^n(V)\otimes C$ is the $\k$-module $V\otimes C_{i-n}$.
If we choose $C$ to be a simply connected differential graded $\k$-coalgebra, we get:
\[
\left(S^n(V)\otimes C\right)_i =\left\lbrace \begin{tabular}{ll}
$0$  & $i< n$,\\
$V$ & $i=n$,\\
$0$ & $i=n+1$,\\
$V\otimes C_{i-n}$ & $i\geq n+2$.
\end{tabular} \right.
\]
Thus, around the $n$-th term, the chain complex $S^n(V)\otimes C$ is similar to $S^n(V)$. 
We can therefore modify the homology of a $C$-comodule for a specific degree without modifying the lower degrees. Therefore the arguments of previous section can be applied. As pullbacks in comodules are computed in the underlying category in our cases, we can adapt Lemma \ref{lem: pullback homology of X} to obtain its following coalgebraic version.

\begin{lem}\label{lem: fixing homology for comodules}
Let $\k$ be a commutative ring with global dimension zero. 
Let $C$ be a simply connected coalgebra in $\chfo$.
Let $X$ be any object in $\comod_C(\chfo)$. Let $V$ be any $\k$-module. Let $n\geq 1$ be any integer.
Given a surjective linear map $f_n\colon X_n\rightarrow V$ non-trivial only on $n$th-homology in a splitting $X_n\cong H_n(X)\oplus B_n(X)\oplus B_{n-1}(X)$, there is a $C$-comodule map $f:X\rightarrow S^n(V)\otimes C$, and the pullback comodule $P$ in the following diagram in $\comod_C(\chfo)$:
\[
\begin{tikzcd}
P\pull \ar{r} \ar{d} & D^n(V)\o C\ar{d}\\
X\ar{r} & S^n(V)\o C,
\end{tikzcd}
\]
has homology:
\[
H_i(P)\cong\left\lbrace \begin{tabular}{ll}
$\ker\left(H_n(f)\right)$ & $i=n$,\\
$H_i(X)$ & $i<n$,
\end{tabular} \right.
\]
and we have $P_i=X_i$ for $0\leq i< n-1$ and $i=n$, and $P_{n-1}=X_{n-1}\oplus V$.
\end{lem}

A similar argument can be applied in the simplicial case and we can thus also fix the homotopy, layer by layer, in the comodule context. 
In order to obtain a coalgebraic variation of Lemma \ref{lem: postX in ch}, we need to understand towers.
The limit $\lim^C_n X_n$ of a tower $\{X(n)\}$ in $\comod_C(\M)$ is usually not equivalent to the limit $\lim_n U(X_n)$ in $\M$.

\begin{defi}
Let $\M$ be either $\smodk$ or $\chfo$.
A tower $\{X(n)\}$ in $\M$ \emph{stabilizes in each degree} if for each degree $i\geq 0$, the tower $\{X(n)_i\}$ of $\k$-modules stabilizes for $n \geq i+1$, i.e., for all $n\geq 0$, and all $0 \leq i \leq n$, we have that the map appearing in the tower induce isomorphisms:
\[
X(n+1)_i\cong X(n+2)_i\cong X(n+3)_i \cong\cdots.\]
Let $C$ be a coalgebra in $\M$.
A tower $\{X(n)\}$ in $\comod_C(\M)$ \emph{stabilizes in each degree} if the underlying tower $\{U(X(n))\}$ in $\M$ stabilizes in each degree.
\end{defi}

\begin{rem}\label{rem: dold-kan preserves stabilization in each degree}
A tower $\{ X(n)\}$ in $\smodk$ stabilizes in each degree if and only if $\{\mathsf{N} (X(n)) \}$ in $\chfo$ stabilizes in each degree.
A tower $\{ X(n)\}$ in $\chfo$ stabilizes in each degree if and only if $\{ \Gamma(X(n)) \}$ in $\smodk$ stabilizes in each degree.
\end{rem}

\begin{lem}\label{lem: cofree preserves stabilization of towers}
Let $\M$ be either $\smodk$ or $\chfo$.
Let $\{X(n)\}$ be a tower in $\M$ that stabilizes in each degree. Let $C$ be any object $\M$.
Then the tower $\Big\lbrace X(n)\otimes C\Big\rbrace$ in $\M$ also stabilizes in each degree and we have:
\[\Big(\lim_n X(n)\Big)\otimes C \cong \lim_n \Big(X(n) \otimes C\Big).\]
\end{lem}

\begin{proof}
We prove only the case $\M=\chfo$.
For all $n\geq 0$, and all $0\leq i \leq n$, we have:
\begin{eqnarray*}
\Big( X(n+1)\otimes C \Big)_i & \cong & \bigoplus_{a+b=i} X(n+1)_a \otimes C_b\\
& \cong & \bigoplus_{a+b=i} X(n+2)_a \otimes C_b\\
& \cong & \Big( X(n+2)\otimes C \Big)_i, 
\end{eqnarray*}
as $0\leq a \leq i \leq n$. This argument generalizes in higher degrees and thus shows that the desired tower stabilizes in each degree. 
For all $i\geq 0$, notice that both $\Big(\Big(\lim_n X(n)\Big)\otimes C\Big)_i$ and $ \Big(\lim_n \Big(X(n) \otimes C\Big)\Big)_i$ are isomorphic to $\ds\bigoplus_{a+b=i} X(i+1)_a \o C_b$.
\end{proof}

\begin{cor}\label{cor: stab in each deg induces same limit}
Let $\M$ be either $\smodk$ or $\chfo$.
Let $C$ be a coalgebra in $\M$.
Let $\{X(n)\}$ be a tower in $\comod_C(\M)$ that stabilizes in each degree. 
Then the natural map:
\[
U(\lim^C_n X(n)) \stackrel{\cong}\longrightarrow \lim_n U(X(n))
\]
is an isomorphism in $\M$.
\end{cor}

\begin{proof}
This follows directly from Lemma \ref{lem: cofree preserves stabilization of towers} as $U$ preserves and reflects a limit precisely when the comonad $-\otimes C:\M\rightarrow \M$ preserves that limit.
In detail, if we write $X:= \lim_n U(X(n))$, then the coaction $X\rightarrow X\otimes C$ is constructed as follows. For each degree $i\geq 0$, the map $X_i\rightarrow (X\otimes C)_i$ is entirely determined by the coaction $X(i+1)\rightarrow X(i+1)\otimes C$.
\end{proof}

Therefore, we can do a coalgebraic variation of Lemma \ref{lem: postX in ch} as the towers stabilize in each degree. Therefore we obtain the following result.

\begin{thm}\label{thm: cocell presentation of comodules over simply connected}
Let $\k$ be commutative ring with global dimension zero.
\begin{enumerate}[label=\upshape\textbf{(\roman*)}]
\item\label{item: Presentation for chains} Let $C$ be a simply connected coalgebra in $\chfo$.
Then $(\iiio \otimes C, \jjj\otimes C)$ is a Postnikov presentation of  $\comod_C(\chfo)$.
\item\label{item: Presentation for simplicial} Let $C$ be a simply connected coalgebra in $\smodk$.
Then $(\iido\otimes C, \jjd\otimes C)$ is a Postnikov presentation of $\comod_C(\smodk)$.
\end{enumerate}
\end{thm}

Using the vocabulary of \cite{HScomonad}, we essentially have shown that the comonad $-\otimes C$ is \emph{tractable} and \emph{allows the inductive arguments} and thus by \cite[5.8]{HScomonad} we indeed have that $(\iiio\o C, \jjj\o C)$ and $(\iido\otimes C, \jjd\otimes C)$ are Postnikov presentations.

An important consequence of this result is that we obtain an  explicit inductive fibrant replacement if we factorize 
the trivial map of right $C$-comodule $X\rightarrow 0$.
We recall that we define homotopy limits of towers as limits of fibrant towers as in Proposition \ref{prop: homotopy theory of towers}.

\begin{cor}\label{SUPER IMPORTANT COR}
Let $\M$ be either $\smodk$ or $\chfo$. Let $C$ be a simply connected coalgebra in $\M$.
Let $X$ be any right $C$-comodule in $\M$.
Then there exists a countable tower $\{X(n)\}$ in $\comod_C(\M)$ with limit $\wt{X}:=\lim_n^C X(n)$ where the right $C$-comodules $X(n)$ are built inductively as follows. 
\begin{itemize}
\item Define $X(0)$ to be the trivial $C$-comodule $0$.
\item Define $X(1)$ to be the cofree $C$-comodule $U(X)\otimes C$. The map $X(1)\rightarrow X(0)$ is trivial.
\item Suppose $X(n)$ was constructed for a certain $n\geq 1$. Then there exists a certain $\k$-module $V_n$ such that $X({n+1})$ is defined as the following pullback in $\comod_C(\M)$:
\begin{itemize}
\item for $\M=\chfo$:
\[
\begin{tikzcd}
X({n+1})\pull \ar{r} \ar{d} & D^{n}(V_n)\o C\ar{d}\\
X(n)\ar{r} & S^{n}(V_n)\o C,
\end{tikzcd}
\]
and we obtain the short exact sequence in $\comod_C(\chfo)$:
\[
\begin{tikzcd}
0 \ar{r} & S^{n-1}(V_n) \otimes C \ar{r} & X(n+1) \ar{r} & X(n) \ar{r} & 0
\end{tikzcd}
\]
\item for $\M=\smodk$: \[
\begin{tikzcd}
X({n+1})\pull \ar{r} \ar{d} & PK(V_n, n)\o C\ar{d}\\
X(n)\ar{r} & K(V_n, n)\o C,
\end{tikzcd}
\]
and we obtain the short exact sequence in $\comod_C(\smodk)$:
\[
\begin{tikzcd}
0 \ar{r} & K(V_n, n-1) \otimes C \ar{r} & X(n+1) \ar{r} & X(n) \ar{r} & 0
\end{tikzcd}
\]
\end{itemize}
\end{itemize}
The tower $\{X(n)\}$ enjoys the following properties.
\begin{enumerate}[label=\upshape\textbf{(\roman*)}]
\item The map $\wt{X}\longrightarrow 0$ is a $(\iiio \otimes C)$-Postnikov tower if $\M=\chfo$ or a $(\iido\otimes C)$-Postnikov tower if $\M=\smodk$. There exists an acyclic cofibration of right $C$-comodules $\begin{tikzcd}[column sep=small] X\ar[hook]{r}{\simeq} & \wt{X}.\end{tikzcd}$

\item If $X$ is a fibrant right $C$-comodule, then $X$ is a retract of $\wt{X}$.

\item For all $n\geq 1$, for all $0\leq i \leq n-1$, we have $H_i(X(n))\cong H_i(X)$ if $\M=\chfo$ or $\pi_i(X(n))\cong \pi_i(X)$ if $\M=\smodk$.

\item The tower $\{X(n)\}$ stabilizes in each degree. In particular: \[U(\wt{X})=U(\lim^C_n X(n))\cong \lim_n (U(X(n))).\]

\item Each map $X({n+1})\rightarrow X(n)$ for $n\geq 0$ is a fibration in $\comod_C(\M)$, and its underlying map $U(X(n+1))\rightarrow U(X(n))$ is also a fibration in $\M$. In particular $\wt{X}$ is the homotopy limit of $\{X(n)\}$ in $\comod_C(\M)$ and we have: $U(\wt{X})\simeq U(\holim^C_n X(n))\simeq \holim_n (U(X(n)))$.
\end{enumerate}
\end{cor}

\begin{defi}\label{def: postnikov tower of comodules}
Let $\M$ be either $\smodk$ or $\chfo$. Let $C$ be a simply connected coalgebra in $\M$.
Let $X$ be a right $C$-comodule in $\M$. 
The \emph{Postnikov tower of $X$} is the tower $\{X(n)\}$ in $\comod_C(\M)$ built in Corollary \ref{SUPER IMPORTANT COR}. 
\end{defi}

\begin{rem}
The Postnikov tower construction is \emph{not} functorial. Given a map of $C$-comodules $X\rightarrow Y$ we do not obtain natural maps $X(n)\rightarrow Y(n)$.
\end{rem}

\subsection{A fibrant-compatible monoidal structure}

We define a symmetric monoidal structure on comodules. For simplicity, we work over cocommutative coalgebras, but one can work instead with non-cocommutative coalgebra and obtain a non-symmetric monoidal structure on bicomodules over the coalgebra.

\begin{defi}
Let $\M$ be either $\smodk$ or $\chfo$.
Let $X$ and $Y$ be $C$-comodules in $\M$. Define their \emph{cotensor product} $X\ccotens Y$ to be the following equalizer in $\M$:
\[
\begin{tikzcd}
X\ccotens Y \ar{r} & X\otimes Y \ar[shift left]{r}\ar[shift right]{r} & X\otimes C \otimes Y,
\end{tikzcd}
\]
where the two parallel morphisms are induced by the coactions $X\rightarrow X\otimes C$ and $Y\rightarrow C\otimes Y$. Since the tensor product preserves equalizers, we obtain a $C$-coaction on $X\ccotens Y$.
\end{defi}





In general, the functor $X\ccotens -\colon \comod_C(\M)\rightarrow \comod_C(\M)$ does not preserve non-finite limits and is neither a left nor a right adjoint. Nevertheless, the cotensor product behaves well with the Postnikov towers of Definition \ref{def: postnikov tower of comodules}.

\begin{lem}\label{lem: postnikov of cotensor}
Let $\M$ be either $\smodk$ or $\chfo$. Let $C$ be a simply connected coalgebra in $\M$.
Let $\{ X(n) \}$ be a Postnikov tower of a $C$-comodule $X$. Let $Y$ be any $C$-comodule. Then $\left\lbrace X(n) \ccotens Y \right\rbrace$ stabilizes in each degree and: \[(\lim_n^C X(n)) \ccotens Y \cong \lim_n^C (X(n)\ccotens Y).\]
\end{lem}

\begin{proof}
Equalizers of towers that stabilize in each degree also stabilize in each degree. 
Then the result follows from Lemma \ref{lem: cofree preserves stabilization of towers}.
\end{proof}

The following result was proved in \cite{connectivecomod} for $\M=\chfo$. We claim we can adapt the arguments to obtain the same results for $\M=\smodk$. The simply connectedness of the coalgebra is again crucial here for Theorem \ref{thm: derived cotensor product} below to show that the cotensor product preserves weak equivalences (using an Eilenberg-Moore spectral sequence argument, see \cite[4.8]{connectivecomod}).


\begin{thm}[{\cite[4.2]{connectivecomod}}]\label{thm: derived cotensor product}
Let $\M$ be either $\smodk$ or $\chfo$. Let $C$ be a simply connected cocommutative coalgebra in $\M$.
Then $(\comod_C(\M), \ccotens, C)$ is a monoidal category with a fibrant-compatible symmetric monoidal structure with fibrant unit.
\end{thm} 

Let $\M$ be either $\smodk$ or $\chfo$.
From Proposition \ref{prop: compatibility of DK in sym mon}, we see that we can (right) derive the cotensor product of comodules and endow the homotopy category of $C$-comodules with a symmetric monoidal structure. More generally, we can endow a symmetric monoidal structure to the underlying $\infty$-category of $\M$. From \cite[1.1, 1.2]{connectivecomod}, this endows a symmetric monoidal structure to the $\infty$-category of homotopy coherent comodules over $C$ in $\mathcal{D}^{\geq 0}(\k)$, the $\infty$-category of connective comodules over the Eilenberg-Mac lane spectrum $H\k$. A priori, the structure depends on the choice of $\M$ as the cotensor products of simplicial comodule and connective differential graded  comodules are not isomorphic. The next section will now prove the monoidal structure is independent from the choice of $\M$.

\section{Application: Dold-Kan Correspondence For Comodules}\label{section: dold-kan correspondence}

We prove here in Theorem \ref{thm: dold-kan correpsondance for Comodules} that the Dold-Kan correspondence lifts to a Quillen equivalence on the categories of (right) comodules, over a simply connected coalgebra.
Moreover, we show in Theorem \ref{thm: comonoidal Dold-Kan comomules eq} that the derived cotensor products induced by Theorem \ref{thm: derived cotensor product} are equivalent. A crucial tool in our arguments is the Postnikov tower (Definition \ref{def: postnikov tower of comodules}) as a fibrant replacement for our model categories of comodules. As a consequence, we obtain a new description of rational $A$-theory in Corollary \ref{cor: rational A-theory}.


\subsection{Lifting the Dold-Kan correspondence}\label{subsection: lifting dold kan}
We dualize the construction appearing in \cite[3.3]{monmodSS}.
Let $L\colon(\C, \o, \bI)\rightarrow (\D, \sm, \bJ)$ be an oplax monoidal functor between symmetric monoidal categories, with right adjoint $R$. 
Let $C$ be a coalgebra in $\C$.
Then the functor $L$ lifts to (right) comodules:
\[
L\colon \comod_C(\C)\longrightarrow \comod_{L(C)}(\D).
\]
However, if we suppose $R$ to only be lax monoidal, then $R$ does not lift to comodules.
Suppose equalizers exist.
We define the correct right adjoint $R^C$ on comodules by the following commutative diagram of adjoints:
\[
\begin{tikzcd}[column sep= huge, row sep=huge]
\comod_C(\C) \ar[shift left=2]{r}{L}[swap]{\perp}  \ar[shift left=2]{d}{U}[swap]{\vdash}& \comod_{L(C)}(\D), \ar[shift left=2]{l}{R^C} \ar[shift left=2]{d}{U}[swap]{\vdash}\\
\C \ar[shift left=2]{r}{L}[swap]{\perp}  \ar[shift left=2]{u}{-\otimes C} & \D.  \ar[shift left=2]{l}{R}[swap]{\perp}  \ar[shift left=2]{u}{-\sm L(C)}
\end{tikzcd}
\] 
In particular, for any object $M$ in $\D$, we define:
\[
R^C(M\sm L(C))= R(M) \otimes C.
\]
The definition of the functor $R^C$ is extended to any $L(C)$-comodule $X$ in $\D$ using the following equalizer in $\comod_{C}(\C)$:
\[
\begin{tikzcd}
R^C(X) \ar{r} & R(X)\otimes C \ar[shift left]{r} \ar[shift right]{r} & R(X\sm L(C))\otimes C.
\end{tikzcd}
\]
We have omitted the forgetful functor $U$ for simplicity.
The first parallel morphism in the equalizer is induced by the coaction $X\rightarrow X\sm L(C)$. 
The second parallel morphism is induced by the universal property of cofree $C$-comodules on the map in $\C$:
\[
\begin{tikzcd}
R(X)\otimes C \ar{r} & R(X)\otimes RL(C) \ar{r} & R(X\sm L(C)),
\end{tikzcd}
\]
where the first map is induced by the unit of adjunction between $L$ and $R$, and the second map is induced by the lax monoidal structure of $R$.

We apply the above construction to the  adjunctions (\ref{eq: dold-kan sym mon}) and (\ref{eq: dold-kan lax mon}) of the Dold-Kan correspondence.  

\begin{thm}\label{thm: dold-kan correpsondance for Comodules}
Let $C$ be a simply connected coalgebra in $\chfo$. Let $D$ be a simply connected  coalgebra in $\smodk$.
Then the induced adjunctions:
\begin{equation}\label{eq: comonoidal DK 1}
\begin{tikzcd}[column sep=large]
\comod_C(\chfo) \ar[shift left=2]{r}{\Gamma} & \ar[shift left=2]{l}{\NN^C}[swap]{\perp} \comod_{\Gamma(C)}(\smodk),
\end{tikzcd}
\end{equation}
and:
\begin{equation}\label{eq: comonoidal DK 2}
\begin{tikzcd}[column sep=large]
\comod_D(\smodk) \ar[shift left=2]{r}{\mathsf{N}} & \ar[shift left=2]{l}{\Gamma^D}[swap]{\perp}\comod_{\mathsf{N}(D)}(\chfo),
\end{tikzcd}
\end{equation}
are Quillen equivalences.
\end{thm}

\begin{rem}\label{rem: Dold-kan can be easy}
Using the rigidification result of \cite{connectivecomod}, we can prove Theorem \ref{thm: dold-kan correpsondance for Comodules} using $\infty$-categories. However, we believe there is value to present a proof that does not depend on the rigidification, and in any cases, both rely on some version of Corollary \ref{SUPER IMPORTANT COR}. Nevertheless, let us explain how to prove Theorem \ref{thm: dold-kan correpsondance for Comodules} using $\infty$-categories. Let $\W_\mathsf{dg}$ be the class of quasi-isomorphisms in $\chfo$. Let $\W_\Delta$ be the class of weak homotopy equivalences in $\smodk$. Let $C$ be a simply connected coalgebra in $\chfo$. Denote by $\W_{\mathsf{dg}, \comod}$ the class of quasi-isomorphisms between differential graded $C$-comodules. Let $\W_{\Delta, \comod}$ denote the class of weak homotopy equivalences between simplicial $\Gamma(C)$-comodules. Given $\Cinf$ a symmetric monoidal $\infty$-category, and $T$ an $\mathbb{A}_\infty$-coalgebra in $\Cinf$, we denote by $\comodinf_T(\Cinf)$ the $\infty$-category of right $T$-comodules in $\Cinf$, i.e., right $T$-modules in the opposite category $\Cinf\op$.  We obtain the following diagram of $\infty$-categories:
\[
\begin{tikzcd}[row sep=large]
\comodinf_C\left( \N(\chfo) \left[\W^{-1}_\mathsf{dg}\right] \right) \ar{r}{\simeq}[swap]{\Gamma} & \comodinf_{\Gamma(C)} \left( \N(\smodk) \left[ \W^{-1}_\Delta\right] \right)\\
\N\left(\comod_C(\chfo) \right) \left[\W_{\mathsf{dg}, \comod}^{-1}\right]\ar{u}{\simeq} \ar{r}{\Gamma} & \N\left(\comod_{\Gamma(C)}(\smodk) \right) \left[ \W_{\Delta, \comod}^{-1} \right]. \ar{u}[swap]{\simeq}
\end{tikzcd}
\]
The vertical maps are equivalences of $\infty$-categories by \cite[1.1]{connectivecomod}. The top horizontal map is an equivalence by \cite[4.1]{coalginDK}: it states that the Dold-Kan correspondence induces a symmetric monoidal equivalence on the underlying $\infty$-categories of $\smodk$ and $\chfo$. In particular, homotopy coherent comodules are equivalent. Therefore, the bottom horizontal map is an equivalence of $\infty$-categories. In particular, the left adjoint of the adjunction (\ref{eq: comonoidal DK 1}) is a Quillen equivalence. The proof is similar for the other adjunction.
\end{rem}

\begin{proof}[Proof of Theorem \ref{thm: dold-kan correpsondance for Comodules}]
We shall only prove the statements for the adjunction (\ref{eq: comonoidal DK 1}). The other one has an entirely similar proof.

Since  the functor $\Gamma\colon \chfo\rightarrow \smodk$ is left Quillen, and cofibrations and weak equivalences are determined by the underlying categories, the left adjoint functor $\Gamma\colon\comod_C(\chfo)\rightarrow \comod_{\Gamma(C)}(\smodk)$ is also left Quillen.
It also reflects weak equivalences between (cofibrant) objects.
Thus, by \cite[1.3.16]{hovey} we only need to show that the counit:
\[
\Gamma \NN^C(X) \longrightarrow X,
\]
is a weak homotopy equivalence, for any fibrant $\Gamma(C)$-comodule $X$. 

For this matter, let $\{ X(n) \}$ be the Postnikov tower of $X$, and denote by $\wt{X}$ the (homotopy) limit of the tower, see Corollary \ref{SUPER IMPORTANT COR}.
Then $X$ is a retract of $\wt{X}$ and thus we only need to show the weak homotopy equivalence when applied on $\wt{X}$.
We prove the claim by induction on the tower. 
Clearly the map $\Gamma \NN^C(X(0)) \longrightarrow X(0)$ is a weak homotopy equivalence as it is the trivial map.

Suppose $X$ is a cofree comodule $M\otimes \Gamma(C)$, for $M$ in $\smodk$.
Then the formula:
\[
\NN^C(M\otimes \Gamma(C)) = \NN(M)\otimes C,
\]
combined with the weak homotopy equivalence:
\[
\begin{tikzcd}
\Gamma(\mathsf{N}(M)\otimes C) \ar{r}{\simeq} & \Gamma(\mathsf{N}(M))\otimes \Gamma(C)\cong M\otimes \Gamma(C),
\end{tikzcd}
\]
induced by the fact that $\Gamma\colon \chfo\rightarrow \smodk$ is part of a weak monoidal Quillen equivalence shows that $\Gamma \NN^C(M\otimes \Gamma(C)) \longrightarrow M\otimes \Gamma(C)$ is a weak homotopy equivalence. 
In particular, we obtain that $\Gamma \NN^C(X(1)) \longrightarrow X(1)$ is a weak homotopy equivalence.

Now suppose we have shown the map $\Gamma\NN^C(X(n)) \longrightarrow X(n)$ is a weak homotopy equivalence for some $n\geq 1$.
Since $X(n+1)$ is the homotopy pullback in $\comod_{\Gamma(C)}(\smodk)$ (and in $\smodk)$:
\[
\begin{tikzcd}
X(n+1) \pull \ar{r} \ar[two heads]{d} & P\otimes \Gamma(C)\ar[two heads]{d}\\
X(n) \ar{r} & Q\otimes \Gamma(C),
\end{tikzcd}
\]
for some epimorphism $P\rightarrow Q$ in $\smodk$,
then as $\NN^C$ is a right Quillen functor, it preserves homotopy pullbacks and thus we get the following homotopy pullback in $\comod_C(\chfo)$ (and in $\chfo$):
\[
\begin{tikzcd}
\NN^C(X(n+1)) \pull \ar{r} \ar[two heads]{d} & \mathsf{N}(P)\otimes C\ar[two heads]{d}\\
\NN^C(X(n)) \ar{r} & \mathsf{N}(Q)\otimes C,
\end{tikzcd}
\]
as $\mathsf{N}$ preserves fibrations. 
In particular, notice that the induced maps $\NN^C(X(n+1))\rightarrow \NN^C(X(n))$ are fibrations in $\comod_C(\chfo)$ and $\chfo$. Moreover, the tower $\{\NN^C(X(n))\}$ stabilizes in each degree (as it is the equalizer of towers that stabilize in each degree by Lemma \ref{lem: cofree preserves stabilization of towers}). Thus the homotopy limits of $\{\NN^C(X(n))\}$ in $\comod_C(\chfo)$ and in $\chfo$ are equivalent to $\NN^C(\wt{X})$.

Consider the above pullback in $\chfo$, then as $\Gamma\colon \chfo\rightarrow \smodk$ is a right Quillen functor, we obtain the following homotopy pullback in $\smodk$:
\[
\begin{tikzcd}
\Gamma\NN^C(X(n+1)) \pull \ar{r} \ar[two heads]{d} & \Gamma(\mathsf{N}(P)\otimes C)\ar[two heads]{d}\\
\Gamma\NN^C(X(n)) \ar{r} & \Gamma(\mathsf{N}(Q)\otimes C).
\end{tikzcd}
\]
Applying the weak homotopy equivalences: \[
\begin{tikzcd}
\Gamma(\mathsf{N}(Q)\otimes C)\stackrel{\simeq}\longrightarrow Q\otimes \Gamma(C), &  \Gamma(\mathsf{N}(P)\otimes C)\stackrel{\simeq}\longrightarrow P\otimes \Gamma(C)
\end{tikzcd}
\]
induced again by $\Gamma\colon\chfo\rightarrow \smodk$ being part of a weak monoidal Quillen equivalence, we obtain by \cite[5.2.6]{hovey} the homotopy pullback in $\smodk$:
\[
\begin{tikzcd}
\Gamma\NN^C(X(n+1)) \pull \ar{r} \ar[two heads]{d} &P\otimes \Gamma(C)\ar[two heads]{d}\\
\Gamma\NN^C(X(n)) \ar{r} & Q\otimes \Gamma(C).
\end{tikzcd}
\]
By induction, since $\Gamma{\mathsf{N}^C}(X(n)) \rightarrow X(n)$ is a weak homotopy equivalence, we obtain that $\Gamma{\mathsf{N}^C}(X(n+1))\rightarrow X(n+1)$ is a weak homotopy equivalence.

We conclude that $\Gamma\NN^C(\wt{X}) \longrightarrow \wt{X}$
is a weak homotopy equivalence by noticing that (we omit the underlying functor):
\begin{eqnarray*}
\Gamma\NN^C(\wt{X}) & \simeq & \Gamma\NN^C(\holim^{\Gamma(C)}_n X(n) ) \\
&\simeq & \Gamma(\holim^C_n\NN^C(X(n)))\\
& \simeq &\Gamma(\holim_n \NN^C(X(n)))\\
& \simeq & \holim_n \Gamma\NN^C(X(n))\\
& \simeq & \holim_n X(n) \\
& \simeq & \wt{X}.
\end{eqnarray*}
We have used that the homotopy limits of $\{X(n)\}$ and $\{ \NN^C(X(n))\}$ can be computed in their underlying categories and that $\NN^C$ is a right Quillen functor. We also used that $\Gamma$ preserves homotopy limits.
\end{proof}

\subsection{A weak opmonoidal Quillen equivalence}
Let $C$ be a cocommutative coalgebra in $\chfo$. 
Consider $\Gamma$ as an oplax symmetric monoidal left adjoint functor induced by the Eilenberg-Zilber map, as in the adjunction (\ref{eq: dold-kan sym mon}).
Let $X$ and $Y$ be $C$-comodules. The oplax structure on $\Gamma$ induces a unique dashed map on the equalizers:
\[
\begin{tikzcd}
\Gamma(X\ccotens Y) \ar{r} \ar[dashed]{d} & \Gamma (X\otimes Y) \ar{d} \ar[shift left]{r} \ar[shift right]{r} & \Gamma(X\otimes C \otimes Y)\ar{d}\\
\Gamma(X) \cotens_{\Gamma(C)} \Gamma(Y) \ar{r} & \Gamma(X)\otimes \Gamma(Y) \ar[shift left]{r}\ar[shift right]{r} & \Gamma(X)\otimes \Gamma(C) \otimes \Gamma(Y).
\end{tikzcd}
\]
Thus $\Gamma$ lifts to a functor $\comod_C(\chfo)\rightarrow \comod_{\Gamma(C)}(\smodk)$ that is oplax symmetric monoidal with respect to their cotensor products.

Recall that we denote by $\Dinf^{\geq 0}(\k)$ the $\infty$-category of connective comodules over the Eilenberg-Mac Lane spectrum $H\k$. It is equivalent to the underlying $\infty$-categories of $\smodk$ and $\chfo$.

\begin{thm}\label{thm: comonoidal Dold-Kan comomules eq}
Let $\k$ be a commutative ring with global dimension zero.
Let $C$ be a simply connected cocommutative coalgebra in $\chfo$. 
Then the induced adjunction:
\[
\begin{tikzcd}[column sep=large]
(\comod_C(\chfo), \ccotens, C) \ar[shift left=2]{r}{\Gamma} & \ar[shift left=2]{l}{\NN^C}[swap]{\perp} (\comod_{\Gamma(C)}(\smodk), \cotens_{\Gamma(C)}, \Gamma(C)),
\end{tikzcd}
\]
is a weak opmonoidal Quillen equivalence. In particular, we obtain an equivalence of symmetric monoidal $\infty$-categories:
\[
\comodinf_C(\Dinf^{\geq 0}(\k)) \simeq \comodinf_{\Gamma(C)}(\Dinf^{\geq 0}(\k)),
\]
with respect to their derived cotensor product of comodules. 
\end{thm}

\begin{rem}
From the adjunction (\ref{eq: dold-kan lax mon}), we can endow $\nn$ with a oplax monoidal left adjoint structure and repeat our above arguments. However, if $D$ is a cocommutative coalgebra in $\smodk$, then there is no hope that $\nn(D)$ is also cocommutative. Thus $\comod_{\nn(D)}(\chfo)$ is not a symmetric monoidal category. Thus we cannot promote the adjunction (\ref{eq: comonoidal DK 2}) of Theorem \ref{thm: dold-kan correpsondance for Comodules} to be a weak opmonoidal Quillen equivalence. Alternatively, we can consider bicomodules over non-cocommutative coalgebras $D$ and $\nn(D)$. In this case we obtain by similar arguments a weak opmonoidal (but not symmetric) equivalence. For simplicity, we do not give details but the proof is entirely similar as of Theorem \ref{thm: comonoidal Dold-Kan comomules eq}.
\end{rem}

\begin{proof}[Proof of Theorem \ref{thm: comonoidal Dold-Kan comomules eq}]
We have $\nn^C(\Gamma(C))\cong C$, thus we only need to show that the lax monoidal map:
\[
\nnn(X)\ccotens \nnn(Y) \longrightarrow \nnn(X\cotens_{\Gamma(C)} Y),
\]
is a weak homotopy equivalence for $X$ and $Y$ fibrant $\Gamma(C)$-comodules.
Let $\{ X(n)\}$ be the Postnikov tower of $X$ as in Definition \ref{def: postnikov tower of comodules}. Let $\wt{X}$ be its (homotopy) limit. It is enough to show that: \[\nnn(\wt{X})\ccotens \nnn(Y) \longrightarrow \nnn(\wt{X}\cotens_{\Gamma(C)} Y),\] is a weak homotopy equivalence.

We show inductively that $\nnn(X(n))\ccotens \nnn(Y)\rightarrow \nnn(X(n)\cotens_{\Gamma(C)} Y)$ is a weak homotopy equivalence.
For $n=0$ the case is vacuous as the map is trivial.
For $n=1$, then $X(1)$ is a cofree $\Gamma(C)$-comodule $M\otimes \Gamma(C)$, for some $M$ in $\smodk$.
Then from the formula $\nnn(M\otimes \Gamma(C))\cong \nn(M)\otimes C$, we get:
\[
\nnn(M\otimes \Gamma(C))\ccotens \nnn(Y)\cong \nn(M)\otimes \nnn(Y).
\]
Thus we need to show that the natural map $\nn(M)\otimes \nnn(Y)\rightarrow \nnn(M\otimes Y)$ is a weak homotopy equivalence.

We show this by using a Postnikov argument on $Y$. Let $\{ Y(n) \}$ be its Postnikov tower with (homotopy) limit $\wt{Y}$. We need to show that $\nn(M)\otimes \nnn(\wt{Y})\rightarrow \nnn(M\otimes \wt{Y})$ is a weak homotopy equivalence.
Let us first show that $\nn(M)\otimes \nnn(Y(n))\rightarrow \nnn(M\otimes Y(n))$, by induction.
For $n=0$ the case is vacuous.
For $n=1$, we have that $Y(1)$ is equivalent to $M'\otimes \Gamma(C)$ for some $M'$ in $\smodk$. Then the map: \[\nn(M)\otimes \nnn(M'\otimes \Gamma(C))\rightarrow \nnn(M\otimes M'\otimes \Gamma(C)),\]
becomes simply the map $\nn(M)\otimes \nn(M')\otimes C\rightarrow \nn(M\otimes M')\otimes C$. It is a weak homotopy equivalence as $\nn(M)\o\nn(M')\rightarrow \nn(M\o M')$ is a weak homotopy equivalence as $\nn$ is part of a weak monoidal Quillen adjunction. 
Now suppose that we have shown $\nn(M)\otimes \nnn(Y(n))\rightarrow \nnn(M\otimes Y(n))$ is a weak homotopy equivalence for some $n\geq 1$.
Recall that $Y(n+1)$ can be described as a homotopy pullback both in $\comod_{\Gamma(C)}(\smodk)$ and in $\smodk$:
\[
\begin{tikzcd}
Y(n+1) \ar[two heads]{d} \ar{r} \pull & P\otimes \Gamma(C) \ar[two heads]{d}\\
Y(n) \ar{r} & Q \otimes \Gamma(C),
\end{tikzcd}
\]
for some epimorphism $P\rightarrow Q$ in $\smodk$. Then as the functors $\nn\colon\smodk \rightarrow \chfo$, $M\otimes-\colon\smodk \rightarrow \smodk$, $\nn(M)\otimes-\colon \chfo\rightarrow \chfo$ and $\nnn\colon\comod_{\Gamma(C)}(\smodk)\rightarrow \comod_C(\chfo)$ preserve homotopy pullbacks, we obtain the following two homotopy pullbacks in $\smodk$:
\[
\begin{tikzcd}
\nn(M)\otimes \nnn(Y(n+1)) \ar{r} \ar[two heads]{d} \pull & \nn(M)\otimes \nn(P)\otimes C\ar[two heads]{d}\\
\nn(M)\otimes \nnn(Y(n))\ar{r} & \nn(M) \otimes \nn(Q) \otimes C,
\end{tikzcd}
\]
and:
\[
\begin{tikzcd}
\nnn(M\otimes Y(n+1)) \ar{r} \ar[two heads]{d} \pull & \nn(M\otimes P)\otimes C\ar[two heads]{d}\\
\nnn(M\otimes Y(n)) \ar{r} & \nn(M\otimes Q) \otimes C.
\end{tikzcd}
\]
Then by induction and our above argument on cofree objects, we obtain by \cite[5.2.6]{hovey} that $\nn(M)\otimes \nnn(Y(n+1))\rightarrow \nnn(M\otimes Y(n+1))$ is a weak homotopy equivalence. 
Now we can conclude as:
\begin{eqnarray*}
\nn(M)\otimes \nnn(\wt{Y}) & \simeq & \nn(M) \otimes \nnn(\holim^{\Gamma(C)}_n Y(n))\\ 
& \simeq & \nn(M) \otimes \holim^C_n\nnn(Y(n))\\
& \simeq & \holim_n^C (\nn(M)\otimes \nnn(Y(n))\\
& \simeq & \holim_n^C (\nnn(M\otimes Y(n))\\
& \simeq &\nnn(\holim_n^{\Gamma(C)} (M\otimes Y(n)))\\
& \simeq & \nnn(M\otimes \wt{Y}).
\end{eqnarray*}
We have used the fact that $\nnn$ and $M\otimes-$ and $\nn(M)\otimes-$ preserve towers that stabilize in each degree and they remain fibrant in the sense of Proposition \ref{prop: homotopy theory of towers}.

We have just shown $\nnn(X(1))\ccotens \nnn(Y)\rightarrow \nnn(X(1)\cotens_{\Gamma(C)} Y)$ is a weak homotopy equivalence. Let us show the higher cases by induction. 
Recall that $X(n+1)$ is determined as a homotopy pullback in $\comod_{\Gamma(C)}(\smodk)$ and in $\smodk$:
\[
\begin{tikzcd}
X(n+1) \ar[two heads]{d} \ar{r} \pull & P'\otimes \Gamma(C)\ar[two heads]{d} \\
X(n) \ar{r} & Q'\otimes \Gamma(C),
\end{tikzcd}
\]
for some epimorphism $P'\rightarrow Q'$ in $\smodk$. Since $\nnn$ and $-\ccotens \nnn(Y)$ preserves fibrations that are epimorphisms and pullbacks, we obtain the following two homotopy pullbacks in $\comod_{\Gamma(C)}(\smodk)$:
\[
\begin{tikzcd}
\nnn(X(n+1))\ccotens \nnn(Y) \ar[two heads]{d} \ar{r} \pull & \nn(P')\otimes \nnn(Y) \ar[two heads]{d}\\
\nnn(X(n)) \ccotens \nnn(Y) \ar{r} & \nn(Q') \otimes \nnn(Y),
\end{tikzcd}
\]
and:
\[
\begin{tikzcd}
\nnn(X(n+1)\cotens_{\Gamma(C)} Y) \ar{r} \ar[two heads]{d} \pull & \nnn(P'\otimes Y) \ar[two heads]{d}\\
\nnn(X(n)\cotens_{\Gamma(C)} Y) \ar{r} & \nnn(Q' \otimes Y).
\end{tikzcd}
\]
By induction and our previous argument, this shows $\nnn(X(n+1))\ccotens \nnn(Y) \rightarrow \nnn(X(n+1)\cotens_{\Gamma(C)} Y)$
is a weak homotopy equivalence.

Since $\nnn$ preserves towers that stabilize in each degree and by Lemma \ref{lem: postnikov of cotensor}, we can conclude by the following string of weak homotopy equivalences:
\begin{eqnarray*}
\nnn(\wt{X}) \ccotens \nnn(Y) & \simeq & \nnn(\holim^{\Gamma(C)}_n X(n)) \ccotens \nnn(Y)\\
& \simeq &(\holim_n^C \nnn(X(n)))\ccotens \nnn(Y)\\
& \simeq & \holim_n^C (\nnn(X(n))\ccotens \nnn(Y))\\
& \simeq & \holim_n^C \nnn(X(n)\cotens_{\Gamma(C)} Y)\\
& \simeq & \nnn(\holim_n^{\Gamma(C)} (X(n)\cotens_{\Gamma(C)} Y)\\
& \simeq & \nnn((\holim_n^{\Gamma(C)} X(n)) \cotens_{\Gamma(C)} Y) \\
& \simeq & \nnn(\wt{X} \cotens_{\Gamma(C)} Y). 
\end{eqnarray*}
Thus $\nnn(\wt{X})\ccotens \nnn(Y) \longrightarrow \nnn(\wt{X}\cotens_{\Gamma(C)} Y)$ is a weak homotopy equivalence.
\end{proof}

\begin{rem}
There is an alternative Dold-Kan correspondence for comodules. Let us dualize another construction from \cite{monmodSS}. Let $\ccoalg(\C)$ denote the category of cocommutative coalgebras in a symmetric monoidal category $\C$.
Let $L\colon(\C, \o, \bI)\rightarrow (\D, \sm, \bJ)$ be an oplax symmetric monoidal functor between symmetric monoidal categories, with right adjoint $R$. 
Then $L$ lifts to a functor $L\colon\ccoalg(\C)\rightarrow \ccoalg(\D)$, with a right adjoint $\overline{R}$, that can be constructed when the categories admit equalizers and cofree cocommutative coalgebras.
Let $D$ be a cocommutative coalgebra in $\D$. The counit $L\overline{R}(D)\rightarrow D$ induces an adjunction from the change of coalgebras:
\[
\begin{tikzcd}[column sep=huge]
 \comod_{L\overline{R}(D)}(\D) \ar[shift left=2]{r}[swap]{\perp} & \comod_D(\D). \ar[shift left=2]{l}{-\cotens_D L\overline{R}(D)}
\end{tikzcd}
\]
If we combine this with the previous construction:
\[
\begin{tikzcd}
 \comod_{\overline{R}(D)}(\C) \ar[shift left=2]{r}{L}[swap]{\perp} & \comod_{L\overline{R}(D)}(\D), \ar[shift left=2]{l}{R^{\overline{R}(D)}}
\end{tikzcd}
\]
and apply this in the case of the adjunction from the Dold-Kan correspondence (\ref{eq: dold-kan sym mon}), we obtain:
\[
\begin{tikzcd}
( \comod_{\overline{\nn}(D)}(\chfo), \cotens_{\overline{\nn}(D)}, \overline{\nn}(D)) \ar[shift left=2]{r}[swap]{\perp} & (\comod_D(\smodk), \cotens_D, D), \ar[shift left=2]{l}
\end{tikzcd}
\]
for any simply connected cocommutative coalgebra $D$ in $\smodk$.
Just as in the proof of Theorem \ref{thm: comonoidal Dold-Kan comomules eq}, we can show that the above adjunction is a weak opmonoidal Quillen pair. 
If we suppose $\Gamma\overline{\nn}(D)\rightarrow D$ to be a weak homotopy equivalence, then one can prove that the above adjunction is a weak opmonoidal Quillen equivalence.
However, from \cite{sore3}, we expect $\Gamma\overline{\nn}(D)\rightarrow D$ to not be a weak homotopy equivalence in most cases.
In particular the functor $\overline{\nn}\colon\ccoalg(\smodk)\rightarrow \ccoalg(\chfo)$ does not represent the correct right adjoint of the left derived functor of $\Gamma$ on the $\infty$-categories of $\mathbb{E}_\infty$-coalgebras in the underlying $\infty$-categories of $\smodk$ and $\chfo$. Thus we do not expect the above Quillen pair to be a Quillen equivalence. We conjecture that the problems disappear if we consider simply connected \emph{conilpotent} cocommutative coalgebras.
\end{rem}

\subsection{Application to \textit{A}-theory}
Let $X$ be a simplicial set.
We say it is simply connected if $X_0=*$ and there are no non-degenerate $1$-simplices.
Denote by $\comod_{X_+}$ the category of $X_+$-comodules in pointed simplicial sets $(\mathsf{sSet}_*,  \wedge, S^0)$, where $X_+=X\coprod \{*\}$.
Let $\mathcal{E}_*$ be a generalized reduced homology theory.
Recall we can endow $\mathsf{sSet}$ with a model structure, denoted by $(\mathsf{sSet}_*)_{\mathcal{E}}$ for which weak equivalences are $\mathcal{E}_*$-equivalences and cofibrations are monomorphisms \cite{bouss}, localizing the usual Kan model structure on $\mathsf{sSet}$.
In \cite[4.8]{HSWald}, Hess and Shipley establish a left-induced combinatorial and simplicial model structure on $X_+$-comodules denoted by $(\comod_{X_+})_\mathcal{E}$ via the adjunction:
\[
\begin{tikzcd}
(\mathsf{sSet}_*)_\mathcal{E} \ar[shift left=2]{r}{-\wedge X_+}[swap]{\perp} & (\comod_{X_+})_\mathcal{E}. \ar[shift left=2]{l}{U}
\end{tikzcd}
\]
Recall that $\smod_\k$ takes part in a Quillen adjunction:
\[
\begin{tikzcd}
(\mathsf{sSet}_*)_\textup{Kan} \ar[shift left=2]{r}{\widetilde{\k}[-]}[swap]{\perp} & \smodk \ar[shift left=2]{l}{U}
\end{tikzcd}
\]
where $\widetilde{\k}[X_+]\cong \k[X]$ is the reduced free simplicial module.
As $\widetilde{\k}[A\wedge B]\cong \widetilde{\k}[A]\otimes_\k \widetilde{\k}[B]$, the functor is strong symmetric monoidal, and we thus obtain a functor $\widetilde{\k}[-]\colon \comod_{X_+}\rightarrow \comod_{\k[X]}(\smodk)$. From now on, we denote $\comod_{\k[X]}(\smodk)$ simply as $\comod_{\k[X]}$. We are interested in the case $\k=\mathbb{Q}$.

\begin{thm}
    Suppose $X$ is a simply connected simplicial set.
    There is a Quillen equivalence
    \[
\begin{tikzcd}
(\comod_{X_+})_{H\Q} \ar[shift left=2]{r}{\widetilde{\Q}[-]}[swap]{\perp} & \comod_{\Q[X]}. \ar[shift left=2]{l}{}
\end{tikzcd}
\]
\end{thm}

\begin{proof}
    First, notice that, by \cite[1.3.16]{hovey}, we have a Quillen equivalence 
    \[
\begin{tikzcd}
(\mathsf{sSet}_*)_{H\Q} \ar[shift left=2]{r}{\widetilde{\Q}[-]}[swap]{\perp} & \smod_\Q. \ar[shift left=2]{l}{U}
\end{tikzcd}
\]
Indeed, let $A\rightarrow B$ be a map in $\mathsf{sSet}_*$ such that the induced map on simplicial homotopy groups $\pi_*(\widetilde{\Q}[A])\rightarrow \pi_*(\widetilde{\Q}[B])$ is an isomorphism.
By the Dold-Kan correspondence, we know $\pi_*(\widetilde{\Q}(A))\cong \widetilde{H}_*(|A|; \Q)$, and therefore  $\widetilde{\Q}[-]$ reflects weak equivalences. 
Now let $M$ be any simplicial $\Q$-module. We need to show that the counit $\widetilde{\Q}[M]\rightarrow M$ is a weak homotopy equivalence. From the short exact sequence:
\[
\begin{tikzcd}
 0\ar{r} &  \Q[*] \ar[hook]{r} & \Q[M] \ar{r} & \widetilde{\Q}[M] \ar{r} &  0
\end{tikzcd}
\]
and the $5$-lemma, it is enough to show $\Q[M]\rightarrow M$ is a weak homotopy equivalence.
By the K\"unneth theorem, as $\Q$ is flat as a $\mathbb{Z}$-module, we obtain that
\[
\pi_*(\Q[M])\cong H_*(\mathsf{N}(\Q[M]))\cong  H_*(\Q\otimes_\mathbb{Z} \mathsf{N}(M))\cong H_*(\Q)\otimes_\mathbb{Z} H_*(\mathsf{N}(M))\cong \Q\otimes_\mathbb{Z}\pi_*(M)
\]
We obtain that the induced map $\Q\otimes_\mathbb{Z} \pi_*(M)\rightarrow \pi_*(M)$ is an isomorphism.

By our discussion at the beginning of Section \ref{subsection: lifting dold kan}, we obtain an adjunction
\[
\begin{tikzcd}
\comod_{X_+} \ar[shift left=2]{r}{\widetilde{\Q}[-]}[swap]{\perp} & \comod_{\Q[X]} \ar[shift left=2]{l}{U^X}
\end{tikzcd}
\]
where $U^X$ is defined as the equalizer in $\comod_{X_+}$:
\[
\begin{tikzcd}
    U^X(M) \ar{r} & U(M) \wedge X_+ \ar[shift left]{r} \ar[shift right]{r} & U(M\otimes {\Q}[X])\wedge X_+
\end{tikzcd}
\]
for any ${\Q}[X]$-comodule $M$. If $M=A\otimes {\Q}[X]$, the cofree comodule on $A\in \smod_\Q$, then $U^X(A\otimes {\Q}[X])\cong U(A)\wedge  X_+$.
Applying again \cite[1.3.16]{hovey}, we need to show that the counit $\widetilde{\Q}[U^X(M)]\rightarrow M$ is a weak homotopy equivalence for any fibrant $\Q[X]$-module $M$.
We argue inductively using Corollary \ref{SUPER IMPORTANT COR}, as $\Q[X]$ is simply connected.
First notice that for $M=A\otimes \Q[X]$, we have from what we argued above:
\[
\widetilde{\Q}[U^X(A\otimes \Q[X])]\simeq \widetilde{\Q}[U(A)\wedge X_+]\simeq \widetilde{\Q}[A]\otimes {\Q}[X] \stackrel{\simeq}\rightarrow A\otimes \Q[X].
\]
Let $\{M(n)\}$ be the Postnikov tower of $M$ as in Definition \ref{def: postnikov tower of comodules} and $\widetilde{M}$ its homotopy limit. 
Suppose we have shown $\widetilde{\Q}[U^X(M(n))]\stackrel{\simeq}\rightarrow M(n)$ for some $n\geq 1$.
Recall that $M(n+1)$ is the homotopy pullback in $\comod_{\Q[X]}$:
\[
\begin{tikzcd}
M(n+1) \ar{r} \ar[two heads]{d} \pull & PK(V_n, n)\otimes \Q[X]\ar[two heads]{d}\\
M(n) \ar{r} & K(V_n, n) \otimes \Q[X].
\end{tikzcd}
\]
Since $U^X$ is a right Quillen functor, we obtain the homotopy pullback in $(\comod_{X_+})_{H\Q}$:
\[
\begin{tikzcd}
U^X(M(n+1)) \ar{r} \ar[two heads]{d} \pull & PK(V_n, n)\wedge X_+\ar[two heads]{d}\\
U^X(M(n)) \ar{r} & K(V_n, n) \wedge X_+.
\end{tikzcd}
\]
By \cite[3.12]{HSWald}, we can view $\comod_{X_+}$ as a reflective subcategory of the category of retractive spaces over $X$, via a right Quillen functor denoted by $-\star X$, that preserves and reflects $H\Q_*$-equivalences \cite[3.13]{HSWald}. Therefore, using the formula \cite[3.11]{HSWald}, we obtain a homotopy pullback in $\mathsf{sSet}$:
\[
\begin{tikzcd}
    U^X(M(n+1))\star X \ar[two heads]{d} \ar{r} \pull & PK(V_n, n)\times X \ar[two heads]{d}\\
    U^X(M(n))\star X \ar{r} & K(V_n, n)\times X.
\end{tikzcd}
\]
Since $X$ is simply connected and $K(V_n, n)$ is a loop space, the right vertical fibration is between nilpotent connected spaces, and thus by the Eilenberg-Moore spectral sequence \cite[4.1]{bousspectral}, we obtain the following homotopy pullback in $\smod_\Q$:
\[
\begin{tikzcd}
    \widetilde{\Q}[U^X(M(n+1))\star X] \ar[two heads]{d} \ar{r} \pull &\widetilde{\Q}[PK(V_n, n)]\otimes \Q[X] \ar[two heads]{d}\\
    \widetilde{\Q}[U^X(M(n))\star X] \ar{r} & \widetilde{\Q}[K(V_n, n)]\otimes \Q[X].
\end{tikzcd}
\]
By induction, as $-\star X$ preserves and reflects $H\Q$-equivalences, we can conclude that $\widetilde{\Q}[U^X(M(n+1))]\stackrel{\simeq}\rightarrow U^X(M(n+1))$ is a weak homotopy equivalence.
Because the tower $\{M(n)\}$ is stable in each degree, then so is $\{\widetilde{\Q}[U^X(M(n))\star X]\}$. Therefore we obtain the string of weak homotopy equivalences:
\begin{align*}
    \widetilde{\Q}[U^X(\widetilde{M})] & \simeq \widetilde{\Q}[U^X(\holim_n^{\Q[X]} M(n))]\\
    & \simeq \widetilde{\Q}[ \holim_n^{X_+} U^X(M(n))]\\
    & \simeq \widetilde{\Q}[ \holim_n (U^X(M(n))\star X)]\\
    & \simeq \holim_n \widetilde{\Q}[(U^X(M(n))\star X)]\\
    & \simeq \holim_n \widetilde{\Q}[U^X(M(n))]\\
    & \simeq \holim_n M(n)\\
    & \simeq \widetilde{M}.
\end{align*}
As in the commutative diagram:
\[
\begin{tikzcd}
 \widetilde{\Q}[U^X(M)] \ar{r} \ar{d}[swap]{\simeq} & M \ar{d}{\simeq} \\
    \widetilde{\Q}[U^X(\widetilde{M})] \ar{r}{\simeq} & \widetilde{M} 
\end{tikzcd}
\]
three-out-of-four maps are weak homotopy equivalences, we can conclude $\widetilde{\Q}[U^X(M)]\rightarrow M$ is a weak homotopy equivalence.
\end{proof}

In \cite[1.3]{HSWald}, Hess-Shipley showed that the localized model structure defines a Waldhausen category $(\comod_{X_+})_\mathcal{E}^\textup{hf}$ of homotopically finite $X_+$-comodules and that there is an equivalence of $K$-theory spectra:
\[
A(X; \mathcal{E}_*) \simeq K((\comod_{X_+})_\mathcal{E}^\textup{hf}).
\]

Notice that the normalization of $\mathbb{Q}[X]$ is equivalent to the singular chain complex $C_*(X; \mathbb{Q})$.
By \cite{duggershipley},  the category $\comod_{C_*(X; \mathbb{Q})}^\textup{perf}$ of compact (i.e. perfect) chain complexes with a $C_*(X; \mathbb{Q})$-comodule structure inherits a Waldhausen category structure from the model structure $\comod_{C_*(X; \mathbb{Q})}(\ch^{\geq 0}_\mathbb{Q})$.
Combining our results we obtain the following.

\begin{cor}\label{cor: rational A-theory}
    Let $X$ be a simply connected  simplicial set. 
    There is an equivalence of $K$-theory spectra
    \[
    A(X; H\mathbb{Q}_*) \simeq K(\comod_{C_*(X; \mathbb{Q})}^\textup{perf}).
    \]
\end{cor}

\begin{proof}
Denote by $\comod_{\mathbb{Q}[X]}^\omega$ the Waldhausen category of compact $\mathbb{Q}[X]$-comodules in $\smod_\mathbb{Q}$. 
By \cite[3.7]{duggershipley}, the previous theorem and Theorem \ref{thm: dold-kan correpsondance for Comodules}, we get the following string of equivalences of $K$-theory spectra:
\[
A(X; H\mathbb{Q}_*)\simeq K((\comod_{X_+})_{H\mathbb{Q}}^\textup{hf}) \simeq K(\comod_{\mathbb{Q}[X]}^\omega) \simeq K(\comod_{C_*(X; \mathbb{Q})}^\textup{perf}). \qedhere
\]
\end{proof}

\renewcommand{\bibname}{References}
\bibliographystyle{amsalpha}

\bibliography{biblio}
\end{document}